\documentclass[11pt, reqno]{amsart}
\usepackage{amsmath}
\usepackage{mathrsfs}
\usepackage{yhmath}
\usepackage{amsxtra}
\usepackage{amscd}
\usepackage{amsthm}
\usepackage{amsfonts}
\usepackage{amssymb}
\usepackage{anysize}
\usepackage{bbding}
\usepackage{bbm}
\usepackage{bm}
\usepackage[bookmarks=true]{hyperref}
\usepackage{empheq}
\usepackage{enumerate}
\usepackage{eucal}
\usepackage{fancyhdr}
\usepackage{mathtools,graphics,latexsym}
\usepackage{pifont}
\usepackage{pmgraph}
\usepackage[usenames,dvipsnames]{color}
\usepackage{tikz-cd}

\marginsize{3cm}{3cm}{2.3cm}{2.3cm}

\input prepictex
\input pictex
\input postpictex

\newtheorem{thm}{Theorem}[section]
\newtheorem{cor}[thm]{Corollary}
\newtheorem{lem}[thm]{Lemma}
\newtheorem{prop}[thm]{Proposition}
\newtheorem{conj}[thm]{Conjecture}

\theoremstyle{definition}

\theoremstyle{remark}
\newtheorem{rem}[thm]{Remark}

\numberwithin{equation}{section}

\begin{document}

\newcommand{\thmref}[1]{Theorem~\ref{#1}}
\newcommand{\secref}[1]{Section~\ref{#1}}
\newcommand{\lemref}[1]{Lemma~\ref{#1}}
\newcommand{\propref}[1]{Proposition~\ref{#1}}
\newcommand{\corref}[1]{Corollary~\ref{#1}}
\newcommand{\conjref}[1]{Conjecture~\ref{#1}}
\newcommand{\remref}[1]{Remark~\ref{#1}}
\newcommand{\eqnref}[1]{(\ref{#1})}
\newcommand{\exref}[1]{Example~\ref{#1}}

\DeclarePairedDelimiterX\setc[2]{\{}{\}}{\,#1 \;\delimsize\vert\; #2\,}
\newcommand{\bn}[1]{\underline{#1\mkern-4mu}\mkern4mu }

\newcommand{\nc}{\newcommand}

\nc{\Z}{{\mathbb Z}}
\nc{\Zp}{\Z_+}
\nc{\C}{{\mathbb C}}
\nc{\N}{{\mathbb N}}
\nc{\R}{{\mathbb R}}

\nc{\bi}{\bibitem}
\nc{\wt}{\widetilde}
\nc{\wh}{\widehat}
\nc{\dpr}{{\prime \prime}}
\nc{\ov}{\overline}
\nc{\bd}{\boldsymbol}
\nc{\un}{\underline}

\nc{\al}{\alpha}
\nc{\ga}{\gamma}
\nc{\de}{\delta}
\nc{\ep}{\epsilon}
\nc{\vep}{\varepsilon}
\nc{\La}{\Lambda}
\nc{\la}{\lambda}
\nc{\si}{\sigma}
\nc{\Sig}{{\bd \Sigma}}
\nc{\sig}{{\bd \sigma}}
\nc{\bpi}{{\bd \pi}}

\nc{\mf}{\mathfrak}
\nc{\mc}{\mathcal}
\nc{\cA}{\mc{A}}
\nc{\cB}{\mc{B}}
\nc{\cN}{\mc{N}}
\nc{\cD}{\mc{D}}
\nc{\cG}{\mc{G}}
\nc{\F}{\mc{F}}
\nc{\fB}{{\mf B}}
\nc{\fb}{{\mf b}}
\nc{\fg}{{\mf g}}
\nc{\fh}{{\mf h}}
\nc{\fL}{\mf{L}}
\nc{\fS}{\mf{S}}
\nc{\cS}{\mc{S}}

\nc{\s}{{\bf s}}
\nc{\w}{{\bf w}}
\nc{\x}{{\bf x}}
\nc{\y}{{\bf y}}
\nc{\z}{{\bf z}}
\nc{\et}{{\bd \eta}}
\nc{\zet}{{\bd \zeta}}

\nc{\xxi}{\un{{\bf x}_i}}
\nc{\etj}{\un{{\bd \eta}_j}}
\nc{\yr}{\un{{\bf y}_r}}
\nc{\zets}{\un{{\bd \zeta}_s}}

\nc{\pa}{\partial}
\nc{\pz}{\partial_z}
\nc{\Ber}{{\rm Ber}}
\nc{\Bers}{\Ber^{\bf s}}
\nc{\cdet}{{\rm cdet}}

\nc{\gl}{{\mf{gl}}}
\nc{\I}{\mathbb{I}}

\nc{\bs}{\bar{s}}
\nc{\bz}{\bar{0}}
\nc{\bo}{\bar{1}}
\nc{\hs}{\hat{s}}
\nc{\bfb}{{\bf b}}
\nc{\cL}{{\mc L}}
\nc{\ovga}{\ov{\ga}}

\nc{\xx}{{\mf x}}
\nc{\hf}{\frac{1}{2}}
\nc{\bB}{{\bf B}}

\nc{\mn}{{m|n}}
\nc{\pqmn}{{p+m|q+n}}
\nc{\gls}{\gl_\pqmn}
\nc{\gld}{\gl_d}
\nc{\glmn}{\gl_{m|n}}
\nc{\glpm}{\gl_{p+m}}
\nc{\glsm}{t^{-1} \gls [t^{-1}]}
\nc{\gldm}{t^{-1} \gld [t^{-1}]}
\nc{\fgm}{t^{-1} \fg [t^{-1}]}
\nc{\fzg}{\fz(\wh{\fg})}
\nc{\Lfg}{^L \fg}
\nc{\fhm}{t^{-1} \fh [t^{-1}]}
\nc{\fnpm}{t^{-1} \np [t^{-1}]}
\nc{\fnmm}{t^{-1} \nm [t^{-1}]}
\nc{\ugtr}{U(\fg[t])_\R}
\nc{\ugmr}{U(t^{-1}\fg[t^{-1}])_\R}

\nc{\Bdw}{\cB_d^\w}
\nc{\Bsz}{{\cB_\pqmn^\z}}
\nc{\Bdmu}{\cB_d^\mu}
\nc{\Bdmuw}{\cB_d^{\mu_\w}}
\nc{\Bpm}{\cB_{p+m}^\z}
\nc{\Ldw}{\cL_d^\w}
\nc{\Lsz}{\cL_{\pqmn}^\z}

\nc{\Bgw}{\cB_\fg^\w}
\nc{\Bgdw}{\cB_{\gld}^\w}
\nc{\Bgsz}{{\cB_{\gls}^\z}}
\nc{\Bgmu}{\cB_\fg^\mu}
\nc{\Bgmuw}{\cB_\fg^{\mu_\w}}
\nc{\Agmuz}{\cA_{\fg}^{\mu}(\z)}

\nc{\Vac}{V_{{\rm crit}}}
\nc{\fz}{{\mf z}}
\nc{\fzs}{\fz(\wh{\gl}_\pqmn)}
\nc{\fzd}{\fz(\wh{\gl}_d)}
\nc{\End}{{\rm End}}
\nc{\ev}{{\rm ev}}
\nc{\evz}{\underline{{\rm ev}}_\z}

\nc{\fzpmqn}{\fz(\wh{\gl}_\pqmn)}
\nc{\fzmn}{\fz(\wh{\gl}_{m|n})}

\nc{\opq}{\mathbbm{1}_{p|q}}
\nc{\op}{\mathbbm{1}_p}
\nc{\bb}{{\bd b}}
\nc{\dds}{{\bf d}_{\F}}
\nc{\bbmu}{{\bd b}^\mu}

\nc{\cT}{\mc{T}}
\nc{\T}{{\rm T}}
\nc{\Xl}{{\bd X}_{\! \ell}}
\nc{\Xd}{{\bd X}_{\! d}}
\nc{\Xpm}{{\bd X}_{\! p+m}}
\nc{\Xpqmn}{{\bd X}_{\! p+q+m+n}}

\nc{\spd}{\mf{sp}_{2d}}
\nc{\sod}{\mf{so}_{2d}}
\nc{\sld}{\mf{sl}_d}

\nc{\womg}{\omega}
\nc{\hv}{\hbar^\vee}
\nc{\hc}{\mf f}
\nc{\fgr}{\fg_\R}
\nc{\qc}{S_0}

\nc{\xai}{x^a_i}
\nc{\xaj}{x^a_j}
\nc{\xbi}{x^b_i}
\nc{\xbj}{x^b_j}
\nc{\yar}{y^a_r}
\nc{\yas}{y^a_s}
\nc{\ybr}{y^b_r}
\nc{\ybs}{y^b_s}
\nc{\yai}{y^a_i}
\nc{\ybi}{y^b_i}
\nc{\pyai}{\pa_{\yai}}
\nc{\pybi}{\pa_{\ybi}}
\nc{\pxai}{\pa_{\xai}}
\nc{\pxaj}{\pa_{\xaj}}
\nc{\pxbi}{\pa_{\xbi}}
\nc{\pxbj}{\pa_{\xbj}}
\nc{\pyar}{\pa_{\yar}}
\nc{\pyas}{\pa_{\yas}}
\nc{\pybr}{\pa_{\ybr}}
\nc{\pybs}{\pa_{\ybs}}
\nc{\pxaip}{\pa_{x^a_{i'}}}
\nc{\peajp}{\pa_{\eta^a_{j'}}}
\nc{\eai}{\eta^a_i}
\nc{\eaj}{\eta^a_j}
\nc{\ebi}{\eta^b_i}
\nc{\ebj}{\eta^b_j}
\nc{\zar}{\zeta^a_r}
\nc{\zas}{\zeta^a_s}
\nc{\zbr}{\zeta^b_r}
\nc{\zbs}{\zeta^b_s}
\nc{\zai}{\zeta^a_i}
\nc{\zbi}{\zeta^b_i}
\nc{\pzai}{\pa_{\zai}}
\nc{\pzbi}{\pa_{\zbi}}
\nc{\peai}{\pa_{\eai}}
\nc{\peaj}{\pa_{\eaj}}
\nc{\pebi}{\pa_{\ebi}}
\nc{\pebj}{\pa_{\ebj}}
\nc{\pzar}{\pa_{\zar}}
\nc{\pzas}{\pa_{\zas}}
\nc{\pzbr}{\pa_{\zbr}}
\nc{\pzbs}{\pa_{\zbs}}

\nc{\np}{\mf{n}_+}
\nc{\nm}{\mf{n}_-}
\nc{\smu}{S^\mu}
\nc{\smuw}{S^{\mu_\w}}
\nc{\sw}{S^{\w}}
\nc{\zpz}{\blb z^{-1}, \pz^{-1} \brb}
\nc{\dsz}{[\tns[ z^{-1} ]\tns]}

\nc{\blb}{(\!(}
\nc{\brb}{)\!)}
\nc{\lb}{\left(\!}
\nc{\rb}{\! \right)}
\nc{\ls}{\left[}
\nc{\rs}{\! \right]}
\nc{\llangle}{\big\langle}
\nc{\rrangle}{\big\rangle}
\nc{\disp}{\displaystyle}

\nc{\ns}{\hspace{-0.3mm}}
\nc{\nns}{\hspace{-2mm}}
\nc{\tns}{\kern-1.1pt}

\advance\headheight by 2pt

\title[Bethe algebras for classical Lie (super)algebras]
{Bethe algebras for unitarizable modules over classical Lie (super)algebras and a duality}

\author[W. K. Cheong]{Wan Keng Cheong}
\address{Department of Mathematics, National Cheng Kung University, Tainan, Taiwan 701401}
\email{keng@ncku.edu.tw}

\author[N. Lam]{Ngau Lam}
\address{Department of Mathematics, National Cheng Kung University, Tainan, Taiwan 701401}
\email{nlam@ncku.edu.tw}

\begin{abstract}

Let $\fg$ denote the classical Lie algebra $\gld$, $\spd$, or $\sod$ with a fixed $*$-structure $\sig$.
Let $M_1, \ldots, M_\ell$ be unitarizable $\fg$-modules (with respect to $\sig$), and let $\z=(z_1, \ldots, z_\ell) \in \C^\ell$.
We investigate the action of the Bethe algebra $\Bgmu$ for $\fg$ with respect to $\mu \in \fg^*$ on the tensor product $\bn M(\z):=M_1(z_1) \otimes \cdots \otimes  M_\ell(z_\ell)$ of evaluation $\fg[t]$-modules.
We show that if $\mu \circ \sig$ equals the complex conjugation of $\mu$, then $\Bgmu$ is diagonalizable on any finite-dimensional $\Bgmu$-submodule of $\bn M(\z)$ for $\z \in \R^\ell$. 
This, together with the result derived from the duality of Bethe algebras (see below), suggests that a simple spectrum conjecture for $\Bgmu$ should hold.

We establish a duality of Bethe algebras for the general linear Lie (super)algebras $\gld$ and $\gls$.
As an application, we show that under a generic condition, the Bethe algebra for $\gls$ with respect to $\z \in \C^{p+q+m+n}$ is diagonalizable with a simple spectrum on any weight space of $L_1(w_1) \otimes \cdots \otimes L_d(w_d)$, where the $L_i$ are (infinite-dimensional) unitarizable highest weight $\gls$-modules corresponding to generalized partitions of depth 1, and $w_1, \ldots, w_d \in \C$. 
We also obtain the corresponding result for $\gl_{p+m}$ by setting $q=n=0$.

\end{abstract}

\maketitle

\setcounter{tocdepth}{1}

\section{Introduction}

Let $\fg$ be a finite-dimensional semisimple Lie algebra over $\C$.
The Bethe algebra for $\fg$ is a commutative subalgebra of $U(\fg[t])$ induced by the Feigin--Frenkel center (see \cite{FFR, FFTL, Ry06}), where $U(\fg[t])$ denotes the universal enveloping algebra of the current algebra $\fg[t]$. 
It produces (higher) Gaudin Hamiltonians, including the quadratic ones originally introduced by Gaudin \cite{G76, G83} to describe a completely integrable quantum spin chain model.
Also, the notion of the Bethe algebra can be extended to the general linear Lie (super)algebras (see \cite{HM, MR, MTV06, Ta}).

Let $d$ be a positive integer, and let $p$, $q$, $m$, and $n$ be non-negative integers that are not all zero.
The aim of the paper is two-fold.
Firstly, we investigate the action of the Bethe algebra for $\fg$ on unitarizable $\fg$-modules, where $\fg$ is the classical Lie algebra of type $\mf{a,c,d}$ of rank $d$ (i.e., $\fg=\gld$, $\spd$, or $\sod$).
Secondly, we establish a duality of Bethe algebras for the general linear Lie (super)algebras $\gld$ and $\gls$ and use it to study the action of the Bethe algebra for $\gls$ on some infinite-dimensional unitarizable $\gls$-modules.

\subsection{Unitarizable modules and diagonalization of Bethe algebras}

Let $\fg$ be a simple Lie algebra or the general linear Lie algebra over $\C$.
For $\mu \in \fg^*$, let $\Bgmu$ be the Bethe algebra for $\fg$ with respect to $\mu$; see \secref{BA}.
The action of $\Bgmu$ on a tensor product of finite-dimensional $\fg$-modules has been extensively studied (see \cite{FFRy, MTV06, MTV08, MTV09-2,MTV09-3,Ry20}).
Let $\z=(z_1, \ldots, z_\ell) \in \C^\ell$, and let $V_1, \ldots, V_\ell$ be finite-dimensional irreducible $\fg$-modules.
Feigin, Frenkel, and Rybnikov \cite{FFRy} show that the algebra $\Bgmu$ acts cyclically on the tensor product $\bn V(\z):=V_1(z_1) \otimes \cdots \otimes  V_\ell(z_\ell)$ of evaluation $\fg[t]$-modules for any regular $\mu$ and any pairwise distinct $z_1, \ldots, z_\ell$.
They also show that $\Bgmu$ is diagonalizable with a simple spectrum on $\bn V(\z)$  for generic $\mu$ and generic $\z$.

Now let $\fg=\gld$, $\spd$, or $\sod$. It possesses a $*$-structure $\sig: \fg \longrightarrow \fg$, i.e., $\sig$ is an even anti-linear anti-involution (see \secref{star}).
Let $M_1, \ldots, M_\ell$ be unitarizable $\fg$-modules with respect to $\sig$, and let $\bn M(\z)=M_1(z_1) \otimes \cdots \otimes  M_\ell(z_\ell)$.
Define $\fg^*_{\sig}=\setc*{\mu \in  \fg^*}{ \mu \circ \sig = \ov\mu}$, where $\ov\mu \in \fg^*$ is the complex conjugation of $\mu$.
Using the properties of $*$-structures and Feigin--Frenkel centers, we obtain the following.

\begin{thm} [\thmref{diag-finite-dim}] \label{main1}
The Bethe algebra $\Bgmu$ is diagonalizable on any finite-dimensional $\Bgmu$-submodule of $\bn M (\z)$ for any $\mu \in \fg^*_{\sig}$ and $\z \in \R^\ell$.
\end{thm}

For $\w:=(w_1, \ldots, w_d) \in \C^d$, let $\mu_\w \in \fg^*$ be the element defined by \eqref{muw1} and \eqref{muw2}.
In other words, $\mu_\w$ corresponds to a diagonal matrix.
We write $\Bgw:=\Bgmuw$.

Let $L_1, \ldots, L_\ell$ be unitarizable highest weight $\fg$-modules with respect to $\sig$.
Set $\un L=L_1 \otimes \cdots \otimes L_\ell$ and $\un L(\z)=L_1(z_1) \otimes \cdots \otimes  L_\ell(z_\ell)$.
Let $\fh$ denote the standard Cartan subalgebra of $\fg$ (see \secref{triang}).
Note that for any weight $\ga \in \fh^*$ of $\un L$, the $\ga$-weight space $\un L(\z)_\ga$ of $\un L(\z)$ is a finite-dimensional $\Bgw$-module, and that $\mu_\w \in \fg^*_{\sig}$ for $\w \in \R^d$.
We have the following as a consequence of \thmref{main1}.

\begin{cor}  [\corref{diag-uni}] \label{maincor}
Let $\ga \in \fh^*$ be any weight of $\un L$. 
Then $\Bgw$ is diagonalizable on $\un L(\z)_\ga$ for any $\w \in \R^d$ and $\z \in \R^\ell$.
\end{cor}

A natural question is whether there are unitarizable analogs of the results of \cite{FFRy} mentioned above. 

\begin{conj} [\conjref{cyclic-diag-conj}] \label{mainconj}
Let $\ga \in \fh^*$ be any weight of $\un L$. 
Then the weight space $\un L(\z)_\ga$ is a cyclic $\Bgw$-module for any pairwise distinct $w_1, \ldots, w_d \in \C$ and any pairwise distinct $z_1, \ldots, z_\ell \in \C$.
Moreover, $\Bgw$ is diagonalizable with a simple spectrum on $\un L(\z)_\ga$ for generic $\w$ and generic $\z$.
\end{conj}

\thmref{main3} (see below) provides positive evidence for \conjref{mainconj} when $q$ and $n$ are set to 0.

\subsection{A duality for the general linear Lie (super)algebras}

For $\w:=(w_1, \ldots, w_d) \in \C^d$ and $\z:=(z_1, \ldots, z_{p+q+m+n}) \in \C^{p+q+m+n}$, we consider the Bethe algebra $\Bdw:=\Bgdw$ for the general linear Lie algebra $\gld$ with respect to $\w$ and the Bethe algebra $\Bsz:=\Bgsz$ for the general linear Lie superalgebra $\gls$ with respect to $\z$; see \secref{BA-gls}.
Note that $\Bdw$ is a subalgebra of $U(\gld[t])$ generated by the coefficients of the series $\cdet (\Ldw)$ while $\Bsz$ is a subalgebra of $U(\gls[t])$ generated by the coefficients of the series $\Bers \big( \Lsz \big)$. Here $\cdet$ and $\Bers$ are respectively the column determinant and the Berezinian of type $\s:=(0^p, 1^q, 0^m, 1^n)$ (see \secref{Ber}), and $\Ldw$ and $\Lsz$ are respectively the matrices defined in \eqref{Ldw} and \eqref{Lsz}.

Let $\F$ be the polynomial superalgebra generated by $\xai$, $\eaj$, $\yar$, and $\zas$, for $i=1, \ldots, m$, $j=1, \ldots, n$, $r=1, \ldots, p$, $s=1, \ldots, q$, and $a=1, \ldots, d$. Here $\xai$ and $\yar$ are even variables while $\eaj$ and $\zas$ are odd.
The superalgebra $\F$ can be realized as the Fock space of $d(p+m)$ bosonic and $d(q+n)$ fermionic oscillators (cf. \cite{CL03, LZ06}).
Let $\cD$ be the corresponding Weyl superalgebra, i.e., $\cD$ is the associative unital superalgebra generated by $\xai$, $\eaj$, $\yar$, and $\zas$ as well as their derivatives $\frac{\pa}{\pa x_i^a}$, $\frac{\pa}{\pa\eta_j^a}$, $\frac{\pa}{\pa y_r^a}$, and $\frac{\pa}{\pa\zeta_r^a}$.
There exist superalgebra homomorphisms $\phi : U(\gld) \longrightarrow \cD$ and $\varphi : U(\gls) \longrightarrow \cD$; see \eqref{phi} and \eqref{ovphi}.
They extend to superalgebra homomorphisms $\phi_\z : U(\gld[t]) \zpz \longrightarrow \cD \zpz$ and $\varphi_\w : U(\gls[t]) \zpz \longrightarrow \cD \zpz$, respectively.
Here, for instance, $\cD \zpz$ is the superalgebra of pseudo-differential operators over $\cD$; see \secref{pseudo}.
The maps $\phi_\z$ and $\varphi_\w$ induce an action of $\Bdw$ and an action of $\Bsz$ on $\F$, respectively.
The following theorem asserts that the actions are equivalent.
We call it the \emph{Bethe duality of $(\gld, \gls)$}.

\begin{thm} [\thmref{Bethe_duality}] \label{main2}
The coefficients of $\phi_\z  \cdet (\Ldw)$ are determined by those of $\varphi_\w  \Bers \big( \Lsz \big)$, and vice versa. Consequently, $\phi_\z \big(\Bdw \big)=\varphi_\w \big(\Bsz \big)$.
\end{thm}

\thmref{main2} recovers several dualities established earlier. After specializing to $p=q=0$, it recovers the Bethe duality of $(\gld, \glmn)$ due to Huang and Mukhin \cite{HM}.
If we assume further that $n=0$ (resp., $m=0$), \thmref{Bethe_duality} gives the Bethe duality of $(\gld, \gl_m)$ obtained by Mukhin, Tarasov, and Varchenko \cite{MTV09} (resp., a variant of the duality of $(\gld, \gl_n)$ obtained by Tarasov and Uvarov \cite{TU}).

Also, another duality of $(\gld, \gl_m)$ is obtained by Vicedo and Young \cite{VY} for Gaudin models with irregular singularities.
It is extended by the authors \cite{ChL25-2} to a duality of $(\gld, \gls)$, which specializes to a variant of \thmref{main2}.

Using the properties of the actions of $\Bdw$ on tensor products of finite-dimensional $\gl_d$-modules, \thmref{main2} enables us to better understand the actions of $\Bsz$ on tensor products of infinite-dimensional unitarizable $\gls$-modules, as shown in the following theorem.

\begin{thm} [\thmref{Bethe-unitary}] \label{main3}
Let $\un L=L_1 \otimes \cdots \otimes L_d$ and $\un L(\w)=L_1(w_1) \otimes \cdots \otimes L_d(w_d)$, 
where the $L_i$ are (infinite-dimensional) unitarizable highest weight $\gls$-modules corresponding to generalized partitions of depth 1 (see \eqref{depth1}).
Let $\mu \in \fh_{\pqmn}^*$ be any weight of $\un L$ (where $\fh_{\pqmn}$ is the Cartan subalgebra of $\gls$ defined in \secref{gl}). 
Then the weight space $\un L(\w)_\mu$ is a cyclic $\Bsz$-module for any pairwise distinct $z_1, \ldots, z_{p+q+m+n} \in \C$ and any pairwise distinct $w_1, \ldots, w_d \in \C$.
Moreover, $\Bsz$ is diagonalizable with a simple spectrum on $\un L(\w)_\mu$ for generic $\z$ and generic $\w$.
\end{thm}

\subsection{Outline}
We organize the paper as follows.
\secref{Pre} is devoted to reviewing the classical Lie (super)algebras of types $\mf{a,c,d}$ and some background material needed in this paper.
In \secref{BA-classical}, we investigate the action of the Bethe algebra for $\fg$ on a tensor product of evaluation $\fg[t]$-modules, where $\fg$ is the classical Lie algebra of type $\mf{a,c,d}$. We prove \thmref{main1} and obtain \corref{maincor}.
In \secref{Duality},  we establish \thmref{main2}, i.e., the Bethe duality of $(\gld, \gls)$.
We give an application of the duality and demonstrate \thmref{main3}, which provides positive evidence for \conjref{mainconj}.
In Appendix \ref{Bers-cdet}, we prove \propref{omega-Ber}, which gives an equality relating $\phi_\z  \cdet (\Ldw)$ to $\varphi_\w  \Bers \big( \Lsz \big)$.

\vskip 0.5cm
\noindent{\bf Notations.}
Throughout the paper, the symbol $\Z$ (resp., $\N$ and $\Zp$) stands for the set of all (resp., positive and non-negative) integers, the symbol $\C$ (resp., $\R$) for the field of complex (resp., real) numbers, and the symbol $\Z_2:=\{\bz, \bo\}$ for the field of integers modulo 2.
 Unless otherwise stated, all vector spaces, algebras, tensor products, etc., are over $\C$.
{\bf We fix $d \in \N$ and $p,q,m,n \in \Zp$.}

\section{Preliminaries} \label{Pre}

In this section, we review the classical Lie (super)algebras of types $\mf{a,c,d}$, evaluation modules, and pseudo-differential operators.

\subsection{Classical Lie (super)algebras} \label{classical}

\subsubsection{The general linear Lie (super)algebra} \label{gl}

For $p, q, m, n\in \Zp$ that are not all zero, let
$$
\I=\{i \in \N \, | \, 1 \le i \le p+q+m+n\}.
$$
For $i \in \I$, define $|i| \in \Z_2$ by
$$
|i|=\begin{cases}
\, \bz & \,\, \mbox{if} \ \  i \in \{1, \ldots, p \} \cup \{ p+q+1, \ldots, p+q+m\},\\
\, \bo & \,\, \mbox{otherwise}.
\end{cases}
$$
Let $\{e_i \, | \, i \in \I\}$ be a basis for the superspace $\C^{p|q}\oplus\C^{m|n}$ such that $\{e_i \, | \, 1\le i \le p+q\}$ and
$\{e_{p+q+i} \, | \, 1\le i \le m+n\}$ are respectively the standard homogeneous bases for $\C^{p|q}$ and $\C^{m|n}$. In other words, the parity of $e_i$ is given by
$|e_i|=|i|$ for $i \in \I$.

For any $i, j \in \I$, let $E^i_j$ denote the $\C$-linear endomorphism on $\C^{p|q}\oplus\C^{m|n}$ defined by
$$
E^i_j (e_k)=\delta_{j, k} e_i \quad \mbox{ for $k \in \I$,}
$$
where $\delta$ is the Kronecker delta. The parity of $E^i_j$ is given by $|E^i_j|=|i|+|j|$.
The superspace of $\C$-linear endomorphisms on $\C^{p|q}\oplus\C^{m|n}$ has a natural structure of a Lie superalgebra, called the \emph{general linear Lie (super)algebra} and denoted by $\gls$. The commutation relations of $\gls$ are given by
$$
[E^i_j, E^k_l]=\de_{j,k} E^i_l-(-1)^{(|i|+|j|)(|k|+|l|)}\de_{i, l} E^k_j \qquad \mbox{for $i, j, k, l \in \I$.}
$$
Note that $\{ E^i_j \, | \, i, j \in \I \}$ is a homogeneous basis for $\gls$.

Let $\mf{b}_{\pqmn}=\bigoplus_{\substack{ i,j \in \I, i \le j} }  \C E^i_j$ be a Borel subalgebra of $\gls$.
The corresponding Cartan subalgebra $\fh_{\pqmn}$ has a basis $\{ E^i_i \, | \, i \in \I \}$.
We denote the dual basis in $\fh_{\pqmn}^{*}$ by $\{ \ep_i \, | \, i \in \I \}$.
The parity of $\ep_i$ is given by $|\ep_i|=|i|$.

For $m=d$ and $p=q=n=0$, we define
\begin{equation} \label{eij-a}
\gld:=\gl_{d|0} \qquad \text{and} \qquad e^{\mf{a}}_{ij}:=e_{ij}:=E^i_j \qquad \mbox{for $i, j=1, \ldots, d$.}
\end{equation}

\subsubsection{The symplectic Lie algebra}

The subalgebra of $\gl_{2d}$ that preserves the nondegenerate skew-symmetric bilinear form on $\C^{2d}$ corresponding to the matrix
$$\left[\begin{matrix}
0 & J_d     \\
-J_d& 0
\end{matrix}\right]$$
is called the \emph{symplectic Lie algebra} and is denoted by $\spd$.
Hereafter, $J_d$ is the $d \times d$ reversion matrix defined by $(J_d)_{ij}=1$ if $i+j=d+1$ and $(J_d)_{ij}=0$ otherwise.

The Lie algebra $\spd$ is simple and is spanned by
\begin{equation} \label{eij-c}
e^{\mf c}_{i j} :=e_{ij}-\vep_i \vep_j e_{j^\prime  i^\prime}
\end{equation}
for $i, j =1, \ldots, 2d$, where $i^\prime:=2d-i+1$, and $\vep_i:=1$ if $1 \le i \le d$ and $\vep_i:=-1$ if $d+1 \le i \le 2d$.

\subsubsection{The orthogonal Lie algebra}

The subalgebra of $\gl_{2d}$ that preserves the nondegenerate symmetric bilinear form on $\C^{2d}$ corresponding to the matrix
$$\left[\begin{matrix}
0 & J_d     \\
J_d& 0
\end{matrix}\right]$$
is called the \emph{orthogonal Lie algebra} and is denoted by $\sod$.
It is spanned by
\begin{equation} \label{eij-d}
e^{\mf d}_{i j} :=e_{i j}-e_{j^\prime  i^\prime}
\end{equation}
for $i, j =1, \ldots, 2d$, where $i^\prime:=2d-i+1$, and is a simple Lie algebra if $d \ge 3$.

\subsubsection{Triangular decompositions} \label{triang}

The general linear Lie (super)algebras $\gls$ and $\gld$ are called the classical Lie (super)algebras of types $\mf{a}$ while the Lie algebras $\spd$ and $\sod$ are called the classical Lie algebras of types $\mf{c}$ and $\mf{d}$, respectively.
Let
\begin{equation*} \label{acd}
\fg^\mf{a}=\gld,\qquad \fg^\mf{c}=\spd, \qquad \fg^\mf{d}=\sod.
\end{equation*}
For $\xx=\mf{a,c,d}$, the classical Lie algebra $\fg^\xx$ has the triangular decomposition
\begin{equation} \label{tri-dec}
\fg^\xx =\nm^\xx \oplus \fh^\xx \oplus \np^\xx,
\end{equation}
where $\fh^\xx : =\bigoplus_{i=1}^d \C e^\xx_{i i}$ is the \emph{standard Cartan subalgebra} of $\fg^\xx$, and $\np^\xx:=\sum_{\substack{i<j} } e^\xx_{ij}$ and $\nm^\xx:=\sum_{\substack{i<j } } \C e^\xx_{ji}$ are nilpotent subalgebras of $\fg^\xx$.
We will suppress the superscript $\xx$ if there is no danger of confusion.

\subsection{Evaluation modules} \label{current}

For any Lie superalgebra $\fg$, we denote by $U(\fg)$ the universal enveloping algebra of $\fg$.
For any even variable $t$, the loop algebra $\fg[t, t^{-1}]:=\fg\otimes \C[t, t^{-1}]$ is defined to be the Lie superalgebra with commutation relations
$$
\hspace{1cm}  \big[A \otimes t^r, B \otimes t^s \big]=[A, B] \otimes t^{r+s} \qquad \mbox{for $A, B \in \fg\,$ and $\, r, s \in \Z$.}
$$
Here $[A, B]$ is the supercommutator of $A$ and $B$.
We will use the notation
$$
\hspace{1cm} A[r]:=A \otimes t^r \qquad \mbox{for $A \in \fg\,$ and $\, r \in \Z$.}
$$
The current algebra $\fg[t]:=\fg \otimes \C[t]$ and the superalgebra $t^{-1} \fg[t^{-1}]:=\fg \otimes t^{-1}\C[t^{-1}]$ are subalgebras of $\fg[t, t^{-1}]$ defined in an obvious way.
We identify $\fg$ with the subalgebra $\fg\otimes 1$ of constant polynomials in $\fg[t]$ and hence $U(\fg) \subseteq U(\fg[t])$.

For $\al \in \C$, we have an \emph{evaluation homomorphism} $\ev_\al: U(\fg[t]) \longrightarrow U(\fg)$ given by
$$
\ev_\al (A[r])= \alpha^r A \qquad \mbox{for $A \in \fg\,$ and $\, r \in \Zp$}.
$$
Any $\fg$-module $M$ is a $\fg[t]$-module via $\ev_\al$, called the \emph{evaluation module} and denoted by $M(\alpha)$.
For any $A \in \fg$ and any even variable $z$, we set
$$
A(z)=\sum_{r=0}^{\infty} A[r]z^{-r-1}.
$$
 It is a formal power series in $z^{-1}$ with coefficients in $\fg[t]$ and acts on $M(\alpha)$ by
$$
A(z) v=\frac{A v}{z-\alpha} \qquad \mbox{for $v \in M(\alpha)$.}
$$
Here each rational function in $z$ represents its power series expression at $\infty$.

Let $\ell \in \N$ and $\z=(z_1, \ldots, z_\ell) \in \C^\ell$. Let $\Delta^{(\ell-1)}: U(\fg[t]) \longrightarrow U(\fg[t])^{\otimes \ell}$ be the $(\ell-1)$-fold coproduct on $U(\fg[t])$.
We have the homomorphism
\begin{equation} \label{evz}
\evz:=(\ev_{z_1} \otimes \ldots \otimes \ev_{z_\ell}) \circ \Delta^{(\ell-1)} : U( \fg[t]) \longrightarrow U(\fg)^{\otimes {\ell}},
\end{equation}
given by
$$
\evz (A[r])=\sum_{i=1}^\ell z_i^r A^{(i)} \qquad \mbox{for $A \in \fg\,$ and $\, r \in \Zp$}.
$$
Here $A^{(i)}:=\underbrace{1\otimes\cdots\otimes1\otimes \stackrel{i}{A}\otimes1\otimes\cdots\otimes1}_{\ell}$ for $i=1, \ldots, \ell$.
We call $\evz$ the \emph{evaluation map at $\z$}.

For any $\fg$-modules $M_1, \ldots, M_\ell$, let
\begin{equation} \label{unM}
\bn M=M_1 \otimes \cdots \otimes M_\ell.
\end{equation}
The tensor product $\bn M$ is a $\fg[t]$-module via $\evz$, which coincides with the tensor product of evaluation $\fg[t]$-modules
\begin{equation} \label{unMz}
\bn M(\z):=M_1(z_1) \otimes \cdots \otimes M_\ell(z_\ell).
\end{equation}
As $\fg$-modules, $\bn M$ and $\bn M(\z)$ are isomorphic via the identity map.
For any $A \in \fg$, the series $A(z)$ acts on $\bn M(\z)$ by
$$
\hspace{1.5cm}  A(z) v=\sum_{i=1}^\ell \frac{A^{(i)}v}{z-z_i} \qquad \mbox{for $v \in \bn M(\z)$}.
$$

\subsection{Pseudo-differential operators} \label{pseudo}

Let $\cA$ be an associative unital superalgebra and $z$ an even variable.
We denote by $\cA [\tns[z]\tns]$ (resp., $\cA \blb z \brb$) the superalgebra of formal power series (resp., formal Laurent series) in $z$ with coefficients in $\cA$.
Let $\cA \zpz$ denote the set of all formal series of the form
$$
  \sum_{j=-\infty}^s \sum_{i=-\infty}^r a_{i j} z^i \pz^j,
$$
where $r, s \in \Z$ and $a_{i j} \in \cA$.
We may endow $\cA \zpz$ with a superalgebra structure using the rules:
\begin{equation}\label{rule}
\pz \pz^{-1}=\pz^{-1} \pz=1, \quad\pz^{i} z^j=\sum_{k=0}^{\infty}  \binom{i}{k} \binom{j}{k} k! \,   z^{j-k} \, \pz^{i-k},  \quad \mbox{for $i, j \in \Z$.}
\end{equation}
Here, e.g., $\disp{\binom{i}{k}:=\frac{i(i-1)\ldots (i-k+1)}{k!}}$.
The superalgebra $\cA \zpz$ is called the superalgebra of pseudo-differential operators over $\cA$.

Let $\omega$ be an anti-involution on $\cA$. In particular, $\omega^2=1$ and $\omega(ab)= \omega(b) \omega(a)$ for $a, b \in \cA$.
We readily see that $\omega$ extends to a $\C$-linear automorphism
$\womg : \cA \zpz \longrightarrow \cA \zpz$
defined by
\begin{equation}
\womg(a z^i \pz^j) =\omega(a) z^j \pz^i
\end{equation}
for $a \in \cA$ and $i, j \in \Z$.
It is straightforward to prove the following.

\begin{prop}\label{omega}
The map $\womg$ is an anti-involution on $\cA \zpz$.
\end{prop}

\section{Bethe algebras for classical Lie algebras} \label{BA-classical}

In this section, we examine the behavior of the Feigin--Frenkel center and the Bethe algebra for the classical Lie algebra $\fg$ of type $\mf{a,c,d}$ under a $*$-structure.
Our primary purpose is to study the diagonalization of the Bethe algebra for $\fg$ on a tensor product of unitarizable $\fg$-modules.

\subsection{Feigin--Frenkel centers} \label{FFC}

Let $t$ be an even variable, and let $\fg$ be a simple Lie algebra or the general linear Lie (super)algebra.
We denote by $\Vac(\fg)$ the {\em vacuum module at the critical level} over the central extension $\wh{\fg}:=\fg[t,t^{-1}]\oplus \C K$ of the loop algebra $\fg[t,t^{-1}]$.
By the Poincar\'e--Birkhoff--Witt theorem, $\Vac(\fg)$ can be identified with $U(\fgm)$ as (super)spaces.
Moreover, the $\wh{\fg}$-module structure on $\Vac(\fg)$ allows us to define a vertex algebra structure on $\Vac(\fg)$, called the {\em universal affine vertex algebra at the critical level} (see \cite{Fr07, FBZ, Mo18, MR} for details).

\sloppy
The center of the vertex algebra $\Vac(\fg)$, denoted by $\fzg$, is called the \emph{Feigin--Frenkel center}.
It is given by
$$
\fzg=\setc*{\! v \in \Vac(\fg)}{\fg[t] v=0 \!}.
$$
Any element of $\fzg$ is called a \emph{Segal--Sugawara vector}.
Also, the (super)algebra $U(\fgm)$ is equipped with the (even) derivation $\T:=-d/dt$ defined by
\begin{equation} \label{der}
\T(1)=0 \qquad \text{and}\qquad \T(A[-r])=r A[-r-1]
\end{equation}
for $A \in \fg$ and $r \in \N$.
Note that $\T$ corresponds to the translator operator on the vertex algebra $\Vac(\fg)$, and that $\fzg$ can be viewed as a commutative subalgebra of $U(\fgm)$ and is $\T$-invariant.

The following result is a famous theorem of Feigin and Frenkel.

 \begin{thm}  [{\cite{FF}}]    \label{FF}
For any simple Lie algebra $\fg$, the Feigin--Frenkel center $\fzg$ has a complete set of Segal--Sugawara vectors $a_1, \ldots, a_d$, where $d$ denotes the rank of $\fg$.
That is, the set $\{\T^r a_i \, | \, i=1, \ldots, d, \, r \in \Zp \}$ is algebraically independent, and $\fzg=\C[\T^r a_i \, | \, i=1, \ldots, d, \, r \in \Zp ]$.
\end{thm}

\begin{rem} \label{CM}
\thmref{FF} holds as well for the general linear Lie algebra (see \cite[Theorem 3.1]{CM} (and also \cite{CT}).
\end{rem}

Let $\fg=\gld$, $\spd$, or $\sod$, which has rank $d$.
For $\xx=\mf{a,c,d}$, recall the elements $e^{\xx}_{ij}$ defined in \eqref{eij-a}, \eqref{eij-c} and \eqref{eij-d}, respectively.
We denote by $\fgr$ the $\R$-subspace of $\fg$ spanned by all the $e^{\mf x}_{ij}$ and by $\fh_\R$ the $\R$-subspace of $\fh$ spanned by all the $e^{\mf x}_{ii}$.

Let $\ugtr$ be the $\R$-subalgebra of $U(\fg[t])$ generated by $A[r]$ for $A \in \fgr$ and $r \in \Zp$, and let $\ugmr$ be the $\R$-subalgebra of $U(\fgm)$ generated by $A[-r]$ for $A \in \fgr$ and $r \in \N$.
The $\R$-algebras $U(\fh[t])_\R$ and $U(t^{-1}\fh[t^{-1}])_\R$ are defined analogously.

By \cite[Theorem 3.1]{CM} for type $\mf{a}$ and \cite[Theorem 2.3]{Mo21} for type $\mf{c,d}$ (see also \cite{Mo13, Ya}), there exist
\begin{equation} \label{cSS}
S_1, \ldots, S_d \in \fzg,
\end{equation}
which admit explicit formulas and form a complete set of Segal--Sugawara vectors for $\fzg$.
An important observation is that $S_1, \ldots, S_d \in \ugmr$.
In the case of type $\mf{a}$, the Segal--Sugawara vectors $S_1, \ldots, S_d$ will be described in detail in \eqref{cdet}.

\subsection{Bethe algebras} \label{BA}

Let $\fg$ be a simple Lie algebra or the general linear Lie algebra.
For any $\mu \in \fg^*$, there is an algebra homomorphism
$$
 \Psi^\mu :  U(\fgm)   \longrightarrow U(\fg[t])  \dsz
$$
given by
 \begin{equation} \label{psi-mu}
  \Psi^\mu(A[-r])=A \otimes (t-z)^{-r}+\de_{r,1} \mu(A), \qquad \mbox{for $A \in \fg \,$ and $\, r \in \N$.}
\end{equation}
 The \emph{Bethe algebra $\Bgmu$ for $\fg$ with respect to $\mu$} is defined to be the subalgebra of $U(\fg[t])$ generated by the coefficients of $\Psi^\mu (S)$ for $S \in \fzg$.
For any $\Bgmu$-module $V$, let $(\Bgmu)_V$ denote the image of the Bethe algebra $\Bgmu$ in the endomorphism algebra $\End(V)$ of $V$. We call $(\Bgmu)_V$  the \emph{Bethe algebra} for $V$ with respect to $\mu$.

Fix $\ell \in \N$ and $\z:=(z_1, \ldots, z_\ell) \in \C^\ell$.
The Feigin--Frenkel center $\fzg$ also induces a distinguished subalgebra of $U(\fg)^{\otimes \ell}$ depending on $\mu$ and $\z$, which we now describe.
The evaluation map $\evz:  U(\fg[t])   \longrightarrow  U(\fg)^{\otimes \ell} $ at $\z$, defined in \eqref{evz}, extends to the homomorphism
 $$
\evz:  U(\fg[t])  \dsz \longrightarrow  U(\fg)^{\otimes \ell} \dsz,
 $$
 such that $\evz(z^{-1})=z^{-1}$.
We have the commutative diagram
$$
\begin{tikzcd}
U(\fgm) \arrow[rd, "\Psi^\mu_\z"'] \arrow[r, "\Psi^\mu"] & U(\fg[t])  \dsz \arrow[d, "\evz"] \\
&  U(\fg)^{\otimes \ell} \dsz
\end{tikzcd}
$$
where $\Psi^\mu_\z$ is defined by
$$
\Psi^\mu_\z(A[-r])=\sum_{i=1}^\ell \frac{A^{(i)}}{(z_i-z)^{r}}+\de_{r,1} \mu(A), \qquad \mbox{for $A \in \fg \,$ and $\, r \in \N$.}
$$
The \emph{Gaudin algebra $\Agmuz$ for $\fg$ with respect to $\mu$ and $\z$} is defined to be the subalgebra of $U(\fg)^{\otimes \ell}$ generated by the coefficients of $\Psi^\mu_\z(S)$ for $S \in \fzg$.
For any $\Agmuz$-module $V$, let $(\Agmuz)_V$ denote the image of $\Agmuz$ in $\End(V)$.
We call $(\Agmuz)_V$  the \emph{Gaudin algebra} for $V$ with respect to $\mu$ and $\z$.

Since $\fzg$ is commutative, so are the algebras $\Agmuz$ and $\Bgmu$.
Also, the actions of $\Agmuz$ and $\Bgmu$ on any tensor product of $\fg$-modules are equivalent.
More precisely, for any $\fg$-modules $M_1, \ldots, M_\ell$, the Gaudin algebra $\Agmuz \subseteq U(\fg)^{\otimes \ell}$ acts on $\bn M$ while the Bethe algebra $\Bgmu \subseteq U(\fg[t])$ acts on $\bn M$ via $\evz$ (see \eqref{unM} and \eqref{unMz} for the notations), and it is evident that
 $$
\big( \Agmuz \big)_{\bn M}=\big(\Bgmu \big)_{\bn M(\z)}.
 $$
In this paper, we will restrict our attention to $\Bgmu$.

 For $\mu \in \fg^*$, we denote by $\fg^\mu$ the centralizer of $\mu$ in $\fg$.
Moreover, $\mu$ is called \emph{regular} if $\fg^\mu$ has dimension equal to the rank of $\fg$.
Set
$$
\Xl=\setc*{\! (z_1, \ldots,z_\ell)\in \C^\ell}{z_i\not=z_j \,\, \mbox{for any $i\not=j$} \!},
$$
called the \emph{configuration space} of $\ell$ distinct points on $\C^\ell$.

Let $M_1, \ldots, M_\ell$ be $\fg$-modules, and let $\z \in \Xl$.
If $\fh \subseteq \fg^\mu$, then $\Bgmu$ commutes with $\fh$ and hence acts on the $\ga$-weight space $\bn M(\z)_\ga$ of $\bn M(\z)$ for any $\ga \in \fh^*$ (\cite[Proposition 4]{Ry06}).

Let $V_1, \ldots, V_\ell$ be finite-dimensional irreducible $\fg$-modules.
We have the following remarkable results about the action of $\Bgmu$ on $\bn V(\z)$.

\begin{thm}  [{\cite[Corollary 5]{FFRy}}]  \label{g-cyclic}
The space $\bn V(\z)$ is a cyclic $\Bgmu$-module for any regular $\mu \in \fg^*$ and $\z \in \Xl$.
\end{thm}

 \begin{thm}  [{\cite[Corollary 6]{FFRy}}]    \label{g-diag}
The Bethe algebra $\Bgmu$ is diagonalizable with a simple spectrum on $\bn V(\z)$ for generic $\mu \in \fg^*$ and generic $\z \in \Xl$.
\end{thm}

\subsection{$*$-structures}  \label{star}

A \emph{$*$-superalgebra} is a pair $(\cA, \sig)$, where $\cA$ is an associative superalgebra, and $\sig: \cA \longrightarrow \cA$ is an even anti-linear anti-involution, called a $*$-structure on $\cA$.
A homomorphism $f: (\cA,\sig) \longrightarrow (\cA^\prime, \sig^\prime)$ of $*$-superalgebras is a homomorphism of superalgebras satisfying $\sig^\prime \circ f = f \circ \sig$.
Let $(\cA,\sig)$ be a $*$-superalgebra, and let $V$ be an $\cA$-module.  A Hermitian form $\langle\cdot|\cdot\rangle$ on $V$ is said to be \emph{contravariant} if $\langle av_1 | v_2 \rangle=\langle v_1 |\sig(a)v_2 \rangle$ for all $a\in \cA$ and $v_1, v_2 \in V$.
An $\cA$-module equipped with a positive definite contravariant Hermitian form is called a \emph{unitarizable} $\cA$-module.

A \emph{Lie superalgebra with $*$-structure} is a pair $(\cG, \sig)$, where $\cG$ is a Lie superalgebra, and $\sig: \cG \longrightarrow \cG$ is an even anti-linear anti-involution, called a $*$-structure on $\cG$.
Note that $\sig$ is a $*$-structure on $\cG$ if and only if the natural extension of $\sig$ to $U(\cG)$ is a $*$-structure on $U(\cG)$.
A homomorphism $f: (\cG, \sig) \longrightarrow (\cG^\prime, \sig^\prime)$ of Lie superalgebras with $*$-structures is a homomorphism of Lie superalgebras satisfying $\sig^\prime \circ f = f \circ \sig$.

Let $(\cG, \sig)$ be a Lie superalgebra with $*$-structure, and let $V$ be a $\cG$-module. A Hermitian form $\langle\cdot|\cdot\rangle$ on $V$ is said to be \emph{contravariant} if $\langle A v_1 |v_2 \rangle=\langle v_1 |\sig(A)v_1 \rangle$ for all $A \in \cG$ and $v_1, v_2 \in V$.
A $\cG$-module equipped with a positive definite contravariant Hermitian form is called a \emph{unitarizable} $\cG$-module.
It is clear that all unitarizable $\cG$-modules are unitarizable $U(\cG)$-modules, and vice versa, and that all unitarizable $\cG$-modules are completely reducible.
We refer the reader to \cite{ChL24, CLZ, LZ06} for related discussions.

The Lie superalgebra $\gls$ admits a $*$-structure $\sig_p: \gls \longrightarrow \gls$ (see \cite[Section 3.2]{CLZ}), defined by

\begin{equation} \label{star-gl}
\hspace{1.5cm} \sum_{i, j \in \I} c_{ij} E^i_j \mapsto  \sum_{i, j \in \I} (-1)^{[i]+[j]} \bar{c}_{ij} E^j_i \qquad \mbox{for $i, j \in \I$},
\end{equation}
where $\bar{c}_{ij}$ denotes the complex conjugate of $c_{ij} \in \C$, and
$$
[i]:=  \begin{cases}
       0& \mbox{if}  \ \  1 \le i \le p,\\
       1 &  \mbox{otherwise}.
     \end{cases}
     $$
After specializing to $q=n=0$, we obtain a $*$-structure
\begin{equation} \label{star-pm}
\sig_p: \gl_{p+m} \longrightarrow \gl_{p+m}.
\end{equation}
Taking $p=m=d$, we obtain a $*$-structure $\sig_d$ on $\gl_{2d}$.
Evidently, the map $\sig_d$ restricts to $*$-structures
\begin{equation} \label{star-type-cd}
\sig_d: \spd \longrightarrow \spd \qquad \text{and} \qquad \sig_d: \sod \longrightarrow \sod.
\end{equation}

We remark that the Lie algebras with $*$-structures $(\gl_{p+m}, \sig_p)$, $(\spd, \sig_d)$, and $(\sod, \sig_d)$ correspond to the Hermitian symmetric spaces of the noncompact type referred to as HS.1, HS.3, and HS.5, respectively, in \cite[p. 37]{ES}.

For the rest of this section, we let $\fg=\glpm$, $\spd$, or $\sod$. Set
\begin{equation} \label{sig}
\sig =
\begin{cases}
\sig_p &\quad \mbox{if $\fg=\glpm$,}\\
\sig_d  &\quad  \mbox{if $\fg=\spd$ or $\sod$}.
\end{cases}
\end{equation}
For $\xx=\mf{a,c,d}$, we clearly have $\sig(e^{\xx}_{ij})=(-1)^{[i]+[j]}e^{\xx}_{ji}$ for all $i,j$. In particular, $\sig(e^{\xx}_{ii})=e^{\xx}_{ii}$ for all $i$.
Note that every non-trivial unitarizable highest weight $\fg$-module with respect to $\sig$ is infinite-dimensional (see \cite{EHW}).

Let $k$ denote the dimension of $\fg$.
Consider the Segal--Sugawara vector
$$
\qc:=\sum_{i=1}^{k} v^i [-1] v_i [-1]\in \fzg,
$$
where $\{v_1, \ldots, v_k \}$ is a basis for $\fg$, and $\{v^1, \ldots, v^k \}$ is the dual basis with respect to a nondegenerate invariant symmetric bilinear form on $\fg$.

\begin{thm} [{\cite[Theorem 1]{Ry08}}]  \label{centralizer}
The Feigin--Frenkel center $\fzg$ is the centralizer of $\qc$ in $U(\fgm)$.
\end{thm}

We denote by $U(\fgm)^\fh$ the $\fh$-centralizer of $U(\fgm)$. It admits a direct sum decomposition
\begin{equation}\label{HC-decom}
U(\fgm)^\fh=U(\fhm) \oplus \mc N,
\end{equation}
where
\begin{eqnarray}
\mc N&:=&U(\fgm)^\fh \cap U(\fgm) \fnmm  \label{N1} \\
&\hspace{1mm}=&U(\fgm)^\fh \cap \fnpm U(\fgm). \label{N2}
\end{eqnarray}
The projection of $U(\fgm)^\fh$ onto $U(\fhm)$ yields an affine version of the Harish-Chandra homomorphism
$$
\hc: U(\fgm)^\fh \longrightarrow U(\fhm).
$$
Note that $\fzg$ can be viewed as a subalgebra of $U(\fgm)^\fh$
while the classical $\mc W$-algebra $\mc W(\Lfg)$ associated to the Langlands dual Lie algebra $\Lfg$ can be viewed as a subalgebra of $U(\fhm)$ (see \cite[Sections 8.1.2--8.1.3]{Fr07} and \cite[Section 13.1]{Mo18}).
Furthermore, the restriction map
$$
\hc: \fzg \longrightarrow \mc W(\Lfg)
$$
is an isomorphism (see \cite[Theorem 8.1.5]{Fr07} and \cite[Theorem 13.1.1]{Mo18}).
It is called the \emph{Feigin--Frenkel isomorphism}.

The $*$-structure $\sig$ on $\fg$, defined in \eqref{sig}, extends naturally to a $*$-structure $\sig: U(\fg) \longrightarrow U(\fg)$, which gives rise to a $*$-structure $\sig: U(\fg[t, t^{-1}]) \longrightarrow U(\fg[t, t^{-1}])$, defined by
$$
\hspace{1.5cm} \sig(A[r])=\sig(A)[r] \qquad \mbox{for $A \in \fg$ and $r \in \Z$}.
$$
Evidently, $\sig$ restricts to $*$-structures on $U(\fgm)$ and on $U(\fg[t])$.

Recall the complete set $\{S_1, \ldots, S_d\}$ of Segal--Sugawara vectors for $\fzg$ mentioned in \eqref{cSS}.
Here we denote the rank of $\glpm$ by $d$.

We have the following proposition, which is proved using the same strategy as the proof of \cite[Proposition 2.4]{Lu}.

\begin{prop} \label{anti-cent}
The $*$-structure $\sig: U(\fgm) \longrightarrow U(\fgm)$ fixes $S_1, \ldots, S_d$.
\end{prop}

\begin{proof}
Let $S \in \{S_1, \ldots, S_d\}$.
Since $\sig \big(\qc \big)=\qc$, we have
$$
\sig(S) \qc=\sig(S) \sig( \qc)=\sig(\qc S)=\sig(S \qc)=\sig(\qc) \sig( S)=\qc \sig(S).
$$
That is, $\sig(S)$ lies in the centralizer of $\qc$. By \thmref{centralizer}, $\sig(S) \in \fzg$.

Recall that $S_1, \ldots, S_d \in \ugmr$. According to \eqref{HC-decom}, $S=H+N$ for some $H \in U(t^{-1}\fh[t^{-1}])_\R$ and $N \in \mc N$.
Clearly, $\sig(H)=H$. By \eqref{N1} and \eqref{N2}, $\sig(N) \in \mc N$. 
Thus, $\hc(\sig(S))=H$, and hence $\hc(\sig(S))=\hc(S)$.
We have $\sig(S)=S$ by the injectivity of $\hc: \fzg \longrightarrow \mc W(\Lfg)$.
\end{proof}

Recall the map $\Psi^\mu :  U(\fgm)   \longrightarrow U(\fg[t])  \dsz$ defined by \eqref{psi-mu}.
For $\mu \in  \fg^*$, let
\begin{equation} \label{cSS-im}
\smu_i(z)=\Psi^\mu (S_i) \qquad \mbox{for $i=1, \ldots, d$.}
\end{equation}
We define $\ov\mu \in  \fg^*$ by $\ov\mu(A)=\overline{\mu(A)}$ for $A \in \fg$.
That is, $\ov\mu$ is the complex conjugation of $\mu$.
Set
$$
 \fg^*_{\sig}=\setc*{\mu \in  \fg^*}{ \mu \circ \sig = \ov\mu}.
$$

\begin{prop} \label{C-inv-Gau-R}
For any $\mu \in  \fg^*_{\sig}$, the $*$-structure $\sig: U(\fg[t]) \longrightarrow U(\fg[t])$ fixes all coefficients of the series $\smu_1(z), \ldots, \smu_d(z)$.
\end{prop}

\begin{proof}
We may extend $\sig: U(\fg[t]) \longrightarrow U(\fg[t])$ to a $*$-structure $\sig: U(\fg[t])  \dsz  \longrightarrow U(\fg[t])  \dsz$ with $\sig(z^{-1})=z^{-1}$.
For any $\mu \in  \fg^*_{\sig}$, the diagram
\begin{eqnarray*} \label{anti-com}
\CD
U(\fgm) @> \Psi^\mu >>U(\fg[t]) \dsz \\
 @V \sig VV @VV \sig V \\
U(\fgm) @> \Psi^\mu  >>U(\fg[t])  \dsz\\
\endCD
\end{eqnarray*}
commutes.
This, together with \propref{anti-cent}, yields $\sig \big( \smu_i(z) \big)=\smu_i(z)$ for $i=1, \ldots, d$.
Thus, the proposition follows.
\end{proof}

\subsection{Unitarizable modules and diagonalization} \label{diagconj}

We will demonstrate the diagonalization of the Bethe algebra for $\fg$ on a tensor product of unitarizable $\fg$-modules under certain conditions.
We need the following well-known result, which follows from \thmref{FF} (see also \remref{CM}) and the fact that $\Psi^\mu ( \T(S) )=\frac{d}{dz} ( \Psi^\mu(S))$ for all $S \in \fzg$.

 \begin{prop} \label{generator}
For any $\mu \in \fg^*$, the Bethe algebra $\Bgmu$ equals the subalgebra of $U(\fg[t])$ generated by the coefficients of $\smu_i(z)$ for $i=1, \ldots, d$.
\end{prop}

Let $M_1, \ldots, M_\ell$ be unitarizable $\fg$-modules (with respect to the $*$-structure $\sig$).
By definition, $M_i$ possesses a positive definite contravariant Hermitian form $\langle\cdot | \cdot\rangle_i$ for each $i=1, \ldots, \ell$.
Let $\langle\cdot | \cdot\rangle$ denote the tensor form of $\langle\cdot | \cdot\rangle_1, \ldots, \langle\cdot | \cdot\rangle_\ell$.
It is given by
$$
\llangle v_1 \otimes \ldots \otimes v_\ell \, | \, v^\prime_1 \otimes \ldots \otimes v^\prime_\ell \rrangle=\prod_{i=1}^\ell \llangle v_i \, | \, v^\prime_i \rrangle_i
$$
for $v_i, v^\prime_i \in M_i$, $1 \le i \le \ell$.

\begin{lem} \label{Hermitian}
Let $\mu \in  \fg^*_{\sig}$ and $\z \in \R^\ell$.
For $i=1, \ldots, d$, let $\bb$ be any coefficient of the series $\smu_i(z)$.
Then $\bb$ is a Hermitian operator on $\bn M(\z)$ with respect to $\langle\cdot | \cdot\rangle$.
\end{lem}

\begin{proof}
Using the contravariance of $\langle\cdot | \cdot\rangle_1, \ldots, \langle\cdot | \cdot\rangle_\ell$ and the hypothesis that $\z \in \R^\ell$, we can check directly that
\begin{eqnarray*}
\llangle A_1[r_1] \ldots A_k[r_k] v\, | \, v^\prime \rrangle \hspace{-2mm}
&=& \hspace{-2mm} \llangle  v\, | \, \sig(A_k)[r_k]  \ldots \sig(A_1)[r_1] v^\prime \rrangle\\
&=& \hspace{-2mm} \llangle  v\, | \, \sig \big(A_1[r_1] \ldots A_k[r_k] \big) v^\prime \rrangle
\end{eqnarray*}
for $v, v^\prime \in \bn M(\z)$, $A_1, \ldots, A_k \in \fg$, $r_1, \ldots, r_k \in \Zp$, and $k \in \N$.
This gives
$$
 \llangle \bb v\, | \, v^\prime \rrangle=\llangle v\, | \, \sig(\bb) v^\prime \rrangle=\llangle v\, | \, \bb v^\prime \rrangle,
$$
where the first equality is due to the sesquilinearity of $\langle\cdot | \cdot\rangle$ and the anti-linearity of $\sig$, and the second equality follows from \propref{C-inv-Gau-R}.
This proves the lemma.
\end{proof}

\begin{thm} \label{diag-finite-dim}
For any finite-dimensional $\Bgmu$-submodule $V$ of $\bn M (\z)$, the Bethe algebra $\Bgmu$ is diagonalizable on $V$ for any $\mu \in  \fg^*_{\sig}$ and $\z \in \R^\ell$.
\end{thm}

\begin{proof}
For $i=1, \ldots, \ell$, let $\bb$ be any coefficient of the series $\smu_i(z)$.
By \lemref{Hermitian}, $\bb$ is a Hermitian operator on $V$ and hence is diagonalizable on $V$.
This proves the theorem in view of \propref{generator}.
\end{proof}

We expect the following conjecture to be true.

\begin{conj} \label{diag-conj}
For any finite-dimensional $\Bgmu$-submodule $V$ of $\bn M (\z)$, the Bethe algebra $\Bgmu$ is diagonalizable on $V$ for generic $\mu \in \fg^*$ and generic $\z \in \C^\ell$.
\end{conj}

Let $\fg=\nm \oplus \fh \oplus \np$ be the triangular decomposition in \eqref{tri-dec}.
For any $\w:=(w_1,\ldots, w_d) \in \C^d$, we define $\mu_\w \in \fg^*$ by
\begin{empheq}[left={\empheqlbrace}]{align}
     &\mu_\w(e^{\xx}_{ii})=-w_i, \quad \mbox{for $1 \le i \le d$;} \label{muw1} \\
    & \mbox{$\mu_\w$ vanishes on $\nm \oplus \np$.}  \label{muw2}
\end{empheq}
That is, $\mu_\w$ corresponds to a diagonal matrix.
The choice of the minus sign will be clear in \secref{BA-gls}.
Set
$$
\Bgw=\Bgmuw.
$$
It is evident that $\fh \subseteq \fg^{\mu_\w}$.
Thus for any weight $\ga \in \fh^*$ of $\bn M$, the weight space $\bn M (\z)_\ga$ is a $\Bgmu$-submodule of $\bn M (\z)$.

\begin{cor} \label{diag-uni}
Let $L_1, \ldots, L_\ell$ be unitarizable highest weight $\fg$-modules (with respect to $\sig$), and let $\ga \in \fh^*$ be any weight of $\un L$.
Then $\Bgw$ is diagonalizable on $\un L(\z)_\ga$ for any $\w \in \R^d$ and $\z \in \R^\ell$.
\end{cor}

\begin{proof}
Since $\un L(\z)_\ga$ is finite-dimensional and $\mu_\w \in  \fg^*_{\sig}$ for $\w \in \R^d$,
the corollary follows from \thmref{diag-finite-dim}.
\end{proof}

\sloppy
\thmref{g-cyclic} and \thmref{g-diag} are statements for finite-dimensional irreducible $\fg$-modules.
Their analogs for unitarizable highest weight $\fg$-modules (with respect to $\sig$) are conjectured to be true.

\begin{conj} \label{cyclic-diag-conj}
Let $\un L$ and $\ga$ be as in \corref{diag-uni}.
Then the weight space $\un L(\z)_\ga$ is a cyclic $\Bgw$-module for any $\w \in \Xd$ and $\z \in \Xl$.
Moreover, $\Bgw$ is diagonalizable with a simple spectrum on $\un L(\z)_\ga$ for generic $\w$ and generic $\z$.
\end{conj}

Positive evidence for \conjref{cyclic-diag-conj} (and hence for \conjref{diag-conj}), in the case where $\fg$ is the general linear Lie algebra, will be given in \corref{evid}.

\subsection{Frobenius algebras}

A \emph{Frobenius algebra} is a finite-dimensional commutative associative unital algebra $\cA$ which admits a nondegenerate symmetric bilinear form $(\cdot, \cdot)$ such that
$$
\hspace{1cm}  (a_1 a_2, a_3)=(a_1, a_2 a_3) \quad \mbox{for all $a_1, a_2, a_3 \in \cA$.}
$$

Let $V_1, \ldots, V_\ell$ be finite-dimensional irreducible $\fg$-modules, and let $\z \in \Xl$.
We will give a brief account of the fact that $(\Bgw)_{\bn V(\z)}$ is a Frobenius algebra for $\w \in \Xd$ (cf. \cite{MTV06, Lu}).

For $\xx=\mf{a,c,d}$, we define an anti-involution $\varpi: \fg \longrightarrow \fg$ by
\begin{equation} \label{Cartan-anti}
\varpi(e^{\mf x}_{ij})=e^{\mf x}_{ji} \qquad \mbox{for all $i, j$.}
\end{equation}
The map $\varpi$ is called the \emph{Cartan anti-involution} on $\fg$.

For $i=1, \ldots, \ell$, $V_i$ is an irreducible highest weight $\fg$-module with a dominant integral highest weight. Let $v_i$ be a highest weight vector of $V_i$.
There is a unique nondegenerate symmetric bilinear form $(\cdot, \cdot)_i$ on $V_i$ defined by
$$
(v_i, v_i)_i=1 \quad \text{and}\quad (A v, v^\prime)_i=(v, \varpi(A) v^\prime)_i
$$
for all $v, v^\prime \in V_i$ and $A \in \fg$ (\cite{Sh}).
The bilinear form $(\cdot, \cdot)_i$ is called the \emph{Shapovalov form} on $V_i$.
It is easy to see that if $v$ and $v^\prime$ are weight vectors of $V_i$ of distinct weights, then $(v, v^\prime)_i=0$.

Let $(\cdot, \cdot)$ denote the tensor form of the bilinear forms $(\cdot, \cdot)_1, \ldots, (\cdot, \cdot)_\ell$, which is called the \emph{tensor Shapovalov form}.
It is a nondegenerate symmetric bilinear form on $\bn V$.
Let $\w \in \C^d$.
By \cite[Theorem 9.1]{MTV06} and \cite[Lemma 2.6]{Lu}, the Bethe algebra $\Bgw$ is symmetric with respect to $(\cdot, \cdot)$, i.e.,
$$
\hspace{1.5cm}  (\bb v, v^\prime)=(v, \bb v^\prime) \quad \mbox{for $\bb \in \Bgw \,$ and $\, v, v^\prime \in \bn V (\z)$}.
$$

 \begin{thm}  \label{Frob}
For any $\w \in \Xd$ and $\z \in \Xl$, the Bethe algebra $(\Bgw)_{\bn V(\z)}$ is a Frobenius algebra. Moreover, $(\Bgw)_{{\bn V(\z)}_{\ga}}$ is a Frobenius algebra for any weight $\ga$ of $\bn V$.
\end{thm}

\begin{proof}
The first statement follows from the above discussion on the tensor Shapovalov form $(\cdot, \cdot)$ on $\bn V$ in view of \thmref{g-cyclic} and \cite[Lemma 2.7]{Lu}. As $(\cdot, \cdot)$ restricts to a nondegenerate symmetric bilinear form on ${\bn V (\z)}_{\! \ga}$ for any weight $\ga$ of $\bn V (\z)$, the second statement follows.
\end{proof}

The following is a consequence of \cite[Lemma 1.3]{Lu}.

\begin{thm} \label{max}

\sloppy
For any $\w \in \Xd$ and $\z \in \Xl$, the following properties hold:

\begin{enumerate} [\normalfont(i)]

\item The Bethe algebra $(\Bgw)_{\bn V(\z)}$ is a maximal commutative subalgebra of $\End(\bn V(\z))$ of dimension $\dim \! \left({\bn V(\z)} \right)$.

\item Every eigenspace of $(\Bgw)_{\bn V(\z)}$ is one-dimensional, and the set of eigenspaces of $(\Bgw)_{\bn V(\z)}$ is in bijective correspondence with the set of maximal ideals of $(\Bgw)_{\bn V(\z)}$.

\item Every generalized eigenspace of $(\Bgw)_{\bn V(\z)}$ is a cyclic $\Bgw$-module.

\end{enumerate}
\end{thm}

\begin{rem}  \label{max-r}
Each statement of \thmref{max} holds if $\bn V(\z)$ is replaced with ${\bn V(\z)}_{\ga}$ for any weight $\mu$ of $\bn V$.
\end{rem}

\section{A Duality of Bethe algebras and an application} \label{Duality}

In this section, we establish a duality of Bethe algebras for the general linear Lie (super)algebras $\gld$ and $\gls$.
We apply it to study the action of the Bethe algebra for $\gls$ on a tensor product of infinite-dimensional unitarizable highest weight $\gls$-modules and obtain evidence for \conjref{cyclic-diag-conj} if $q$ and $n$ are set to 0.

\subsection{Column determinants and Berezinians} \label{Ber}

Let us review the notions of column determinants and Berezinians.

Let $\cA$ be an associative unital superalgebra.
The parity of a homogeneous element $a \in \cA$ is denoted by $|a|$, which lies in $\Z_2$.
Fix $k \in \N$. For any $k \times k$ matrix $A=\big[a_{i,j}\big]_{i,j=1,\ldots,k}$ over $\cA$, the \emph{column determinant} of $A$ is defined to be
$$
\cdet (A)=\sum_{\si\in\mf{S}_k} (-1)^{l(\si)} \,  a_{\si(1),1}\ldots a_{\si(k),k}.
$$
Here $\fS_k$ denotes the symmetric group on $\{1, \ldots, k\}$ and $l(\si)$ denotes the length of $\si$.

Assume that $A$ has a two-sided inverse $A^{-1}=\big[\wt{a}_{i,j}\big]$.
For all $i, j=1, \ldots, k$, the \emph{$(i,j)$th quasideterminant} of $A$ is defined to be $|A|_{ij}:=\wt{a}_{j,i}^{-1}$ provided that $\wt{a}_{j,i}$ has an inverse in $\cA$.
Following the notation of \cite{GGRW}, it is convenient to write
$$
|A|_{ij}=
\begin{vmatrix}
 a_{1,1} 	& \ldots & a_{1,j} &\ldots & a_{1,k}\\
 \ldots	& \ldots & \ldots  &\ldots & \ldots	\\
 a_{i,1}& \ldots & \fbox{$a_{i,j}$} &\ldots &a_{i,k}\\
 \ldots	& \ldots & \ldots  &\ldots & \ldots	\\
 a_{k,1}&\ldots &a_{k,j} &\ldots &a_{k,k}
\end{vmatrix}.
$$
For $i=1,\ldots,k$, we define
$$
d_i(A)=\begin{vmatrix}
	 a_{1,1} 	& \ldots & a_{1,i}\\
	 \ldots	& \ldots & \ldots \\
	 a_{i,1}& \ldots & \fbox{$a_{i,i}$}
\end{vmatrix},\
$$
which are called the \emph{principal quasiminors} of $A$.

Let $A=\big[a_{i,j}\big]_{i,j=1,\ldots,k}$ be a $k \times k$ matrix over $\cA$.
For any nonempty subset $P=\{i_1<\ldots<i_p\}$ of $\{1,\ldots,k\}$, the matrix $A_{P}:=\big[a_{i,j}\big]_{i,j \in P}$ is called a \emph{standard submatrix} of $A$.
We say that $A$ is \emph{sufficiently invertible} if every principal quasiminor of $A$ is well defined, and that $A$ is \emph{amply invertible} if each of its standard submatrices is sufficiently invertible.

A sequence $(s_1,\ldots,s_{m+n})$ of 0's and 1's is called a \emph{$0^m1^n$-sequence} if exactly $m$ of the $s_i$'s are 0 and the others are 1.
Every $0^m1^n$-sequence can be written in the form
$(0^{m_1}, 1^{n_1}, \ldots, 0^{m_r}, 1^{n_r}),$
where the sequence begins with $m_1$ copies of $0$'s, followed by $n_1$ copies of $1$'s, and so on.
The set of all $0^m1^n$-sequences is denoted by $\cS_{\mn}$.

Fix $\s:=(s_1,\ldots,s_{m+n}) \in \cS_{\mn}$.
For any $\si \in \fS_{m+n}$ and any $(m+n)\times (m+n)$ matrix $A=[a_{i,j}]_{i=1,\ldots,m+n}$ over $\cA$, we define $\s^\si=\left(s_{\si^{-1}(1)},s_{\si^{-1}(2)},\ldots,s_{\si^{-1}(m+n)} \right)$ and $A^\si=\big[a_{\si^{-1}(i),\si^{-1}(j)} \big]_{i=1,\ldots,m+n}$.
We say that $A$ is of type $\s$ if $a_{i,j}$ is homogeneous of parity $|a_{i,j}|=\bs_i+\bs_j$ for any $i,j = 1,\ldots,m+n$.

For any sufficiently invertible matrix $A$ of type $\s$  over $\cA$, the \emph{Berezinian of type $\s$} of $A$ is defined to be
$$
\Bers (A)=d_1(A)^{\hs_1}\ldots d_{m+n}(A)^{\hs_{m+n}},
$$
where $\hs_i:=(-1)^{s_i}$ for $i=1, \ldots, m+n$ (see \cite[(3.3)]{HM}). We refer the reader to \cite{Ber, Na, MR} for earlier definitions of Berezinians.
The following proposition is useful.

\begin{prop} [{cf. \cite[Proposition 3.5]{HM}}] \label{decomp}
Let $A$ be an amply invertible matrix of type $\s$ over $\cA$.
Fix $k\in \{1,\ldots,m+n-1\}$. We write
\begin{equation} \label{block}
A=\begin{bmatrix} W & X\\ Y&Z \end{bmatrix},
\end{equation}
where $W,X,Y,Z$ are respectively $k\times k$, $k\times(m+n-k)$, $(m+n-k)\times k$, and $(m+n-k)\times (m+n-k)$ matrices.
Then $W$ and $Z-YW^{-1}X$ are sufficiently invertible matrices of types $\s^\prime:=(s_1,\ldots,s_k)$ and $\s^\dpr:=(s_{k+1},\ldots,s_{m+n})$, respectively. Moreover,
$$
\Bers (A)=\Ber^{\s^\prime} (W) \cdot \Ber^{\s^\dpr} \big(Z-YW^{-1}X \big).
$$
\end{prop}

A matrix $A=[a_{i,j}]_{i=1,\ldots,m+n}$ over $\cA$ is called a \emph{Manin matrix} of type $\s$ if $A$ is a matrix of type $\s$ satisfying the following relations
$$
[a_{i,j}, a_{k,l}]=(-1)^{s_i s_j+s_i s_k+s_j s_k}[a_{k,j}, a_{i, l}]  \qquad  \text{for all}  \quad i,j,k,l = 1,\ldots,m+n,
$$
where $[a, b]:=ab-(-1)^{|a||b|}ba$ for any homogeneous elements $a, b \in \cA$.
We recall some basic facts about Manin matrices.

\begin{prop}[{cf. \cite[Section 3]{HM}}] \label{basics}

Let $A$ be a Manin matrix of type $\s$ over $\cA$. Then

\begin{enumerate}[\normalfont(i)]

\item If $P =\{i_1<\ldots<i_p \}$ is a nonempty subset of $\{1,\ldots,m+n\}$, then the standard submatrix $A_{P}$ is a Manin matrix of type $(s_{i_1},\ldots s_{i_p})$.

\item For any $\si \in \fS_{m+n}$, $A^\si$ is a Manin matrix of type $\s^\si$.

\end{enumerate}

\end{prop}

The following proposition asserts that the Berezinian of a Manin matrix is invariant under any permutation.

\begin{prop} [{\cite[Proposition 3.6]{HM}}] \label{HM}

Let $A$ be an amply invertible Manin matrix of type $\s$ over $\cA$.
Then
$$
 \hspace{1cm}  \Ber^{\s^\si} (A^\si)=\Bers (A)  \qquad \mbox{for any $\si \in \fS_{m+n}$.}
$$
\end{prop}

We may decompose $\Bers (A)$ in \propref{block} in a different way if $A$ is a Manin matrix.

\begin{prop} \label{block2}
Let $A$ be an amply invertible Manin matrix of type $\s$ over $\cA$. Write $A$ as in \eqref{block}.
Then $W-XZ^{-1} Y$ and $Z$ are sufficiently invertible matrices of types $\s^\prime:=(s_1,\ldots,s_k)$ and $\s^\dpr:=(s_{k+1},\ldots,s_{m+n})$, respectively.
Moreover,
\begin{equation} \label{dec2}
\Bers (A)=\Ber^{\s^\dpr} (Z)  \cdot \Ber^{\s^\prime} \big(W-XZ^{-1} Y \big).
\end{equation}
\end{prop}
\begin{proof}
The first assertion is clear.
Let $\si \in \fS_{m+n}$ be the permutation defined by $\si(k+i)=i$ for $i=1, \ldots, m+n-k$ and $\si(j)=m+n-k+j$ for $j=1, \ldots,  k$.
Then
$$
A^\si=\begin{bmatrix} Z & Y \\ X& W \end{bmatrix}.
$$
The equality \eqref{dec2} is now a direct consequence of \propref{decomp} and \propref{HM}.
\end{proof}

We also have the following.

\begin{prop}[{\cite[Lemma 8]{CFR}}] \label{CFR}
Let $A$ be a sufficiently invertible Manin matrix of type $\s_0:=(0^m)$ over $\cA$.
Then $\Ber^{\s_0} (A)=\cdet (A)$.
\end{prop}

\subsection{The Bethe algebras for $\gld$ and $\gls$} \label{BA-gls}

Let $t$ be an even variable and
$$
\tau=-\partial_t.
$$
Here $t$ and $\partial_t$ satisfy the rules similar to \eqref{rule}.
Let $\w=(w_1,\ldots, w_d) \in \C^d$, and recall $\mu_\w \in \gld^*$ defined by \eqref{muw1} and \eqref{muw2}.
Write $$\Psi^\w= \Psi^{\mu_\w}.$$
The map $\Psi^\w$ extends to an algebra homomorphism
$$
\Psi^\w :  U(\gldm) [\tau]  \longrightarrow U(\gld[t])  \dsz [\pz]
$$
such that $\Psi^\w(\tau)=\pz$.

Let
$$
\cT_d = \Big[\delta_{a, b} \tau+e_{ab}[-1] \Big]_{a,b=1,\ldots,d},
$$
which is a Manin matrix of type $(0^d)$.
The vectors $S_1, \ldots, S_d \in \fz(\wh{\gl}_d)$, mentioned in \eqref{cSS}, are the coefficients of $\cdet (\cT_d) \in U(\gldm) [\tau]$; more precisely,
\begin{equation} \label{cdet}
 \cdet (\cT_d)=\tau^d+\sum_{i=1}^{d}S_i \tau^{i-1}
\end{equation}
 (see \cite[Theorem 3.1]{CM}).
Now consider the Manin matrix
\begin{equation} \label{Ldw}
\Ldw := \Big[\Psi^\w \big(\delta_{a, b} \tau+e_{ab}[-1] \big)\Big]_{a,b=1,\ldots,d}
= \Big[\delta_{a,b}(\pz-w_a)-e_{ab}(z)\Big]_{a,b=1,\ldots,d}.
\end{equation}
Then in view of \eqref{cdet},
\begin{equation} \label{cdet_SS}
 \cdet (\Ldw)=\pz^d+\sum_{i=1}^{d}S_i^{\w}(z)\pz^{i-1},
\end{equation}
where $S_i^{\w}(z):= \Psi^\w(S_i)$.
Write
$$
\Bdw=\Bgdw.
$$
Recall from \propref{generator} that $\Bdw$ is equal to the subalgebra of $U(\gld[t])$ generated by the coefficients of $S_i^{\w}(z)$ for $i=1, \ldots, d$.

Fix $\s:=(0^p, 1^q, 0^m, 1^n) \in \cS_\pqmn$.
Let
$$
\cT_\pqmn=\Big[\delta_{i,j}\tau+(-1)^{|i|}E^i_j[-1]\Big]_{i,j\in \I}.
$$
It is straightforward to verify that $\cT_\pqmn$ is an amply invertible Manin matrix of type $\s$ over $U(\glsm) \blb \tau^{-1} \brb$ (cf. \cite[Lemma 3.1]{MR}).
We have the following expansion:
\begin{equation} \label{Bers-T}
\Bers  \! \left(\cT_\pqmn \right)=\sum_{i=-\infty}^{p+m-q-n}b_i \tau^i,
\end{equation}
for some $b_i \in U(\glsm)$.

Let $\wh{\fz}_\pqmn$ be the subalgebra of $U(\glsm)$ generated by
$$
\{\T^r b_i \, | \, i \le p+m-q-n, \, i \in \Z, \, r \in \Zp \},
$$
where $\T$ is defined by \eqref{der}.
According to \cite[Corollary 3.3]{MR} (see also  \cite[Proposition 3.3]{ChL25-1}), $b_i \in \fzpmqn$ for all $i \in \Z$ with $i \leq p+m-q-n$.
Since $ \fzpmqn$ is $\T$-invariant, we see that $\wh{\fz}_\pqmn$ is a commutative subalgebra of $ \fzpmqn$.

There is a conjecture asserting that $\fzpmqn=\wh{\fz}_\pqmn$ (see \cite{MR}).
The conjecture holds if $q=n=0$ (see \cite[Theorem 3.1]{CM} and also \thmref{FF}).
It is also true that $\fz(\wh{\gl}_{1|1})=\wh{\fz}_{1|1}$ (\cite{MM15}) and $\fz(\wh{\gl}_{2|1})=\wh{\fz}_{2|1}$ (\cite{AN}).
However, it is still unknown whether the conjecture is valid in general.

Now let $\z=(z_1,\ldots, z_{p+q+m+n}) \in \C^{p+q+m+n}$.
Define $\mu_\z \in \gls^*$ by
$$
\hspace{1cm}  \mu_\z(E^i_j)=(-1)^{|i|+1} \de_{i,j} z_i \qquad \mbox{for $i, j \in \I$.}
$$
Similar to \eqref{psi-mu}, we have a superalgebra homomorphism
$$
 \Psi^\z :  U(\glsm)   \longrightarrow U(\gls[t]) \blb z^{-1}\brb
$$
given by
 $$
  \Psi^\z(A[-r])=A \otimes (t-z)^{-r}+\de_{r,1} \mu_\z(A), \quad \mbox{for $A \in \gls \,$ and $\, r \in \N$.}
 $$
The map $\Psi^\z$ extends to a superalgebra homomorphism
$$
\Psi^\z :  U(\glsm) \blb \tau^{-1} \brb  \longrightarrow U(\gls[t]) \zpz
$$
given by
 $$
\Psi^\z  \left(\sum_{i=-\infty}^s a_i \tau^i \right)=\sum_{i=-\infty}^s \Psi^\z (a_i) \pz^i
 $$
for $a_i \in U(\glsm)$ and $s \in \Z$.
Consider the matrix
\begin{equation} \label{Lsz}
\Lsz := \Big[\Psi^\z \big(\delta_{i,j}\tau+(-1)^{|i|}E^i_j[-1] \big)\Big]_{i,j\in \I}
= \Big[\de_{i,j}(\pz-z_i)-(-1)^{|i|}E^i_j(z)\Big]_{i,j\in \I}.
\end{equation}
It is an amply invertible Manin matrix of type $\s$ over $U(\gls[t]) \zpz$.
The following is an immediate consequence of \propref{HM}.

\begin{prop} \label{Ber-sig-L}
For each $\si \in \fS_{p+q+m+n}$,
$$
\Ber^{\s^\si} \big(\big(\Lsz\big)^\si \big)=\Bers \big( \Lsz \big).
$$
\end{prop}

By \eqref{Bers-T}, we have
\begin{equation} \label{Bers}
\Bers \big( \Lsz \big)=\sum_{i=-\infty}^{p+m-q-n}b_i^{\z}(z)\pz^i,
\end{equation}
where $b_i^{\z}(z):=\Psi^\z (b_i)$.
Let $\Bsz:=\Bgsz$ be the subalgebra of $U(\gls[t])$ generated by the coefficients of the series $b_i^{\z}(z)$, for $i \in \Z$ with $i \le p+m-q-n$.
It is called the \emph{Bethe algebra} for $\gls$ with respect to $\z$.
We remark that $\Bsz$ coincides with the subalgebra of $U(\gls[t])$ generated by the coefficients of $\Psi^\z(S)$ for $S \in \wh{\fz}_\pqmn$ (as $\Psi^\z ( \T(S) )=\frac{d}{dz} ( \Psi^\z(S))$).
If we take $q=n=0$, we recover the Bethe algebra for $\glpm$ by virtue of \propref{CFR}.

According to \cite[Corallary 3.7]{MR}, the algebra $\Bsz$ is a commutative subalgebra of $U(\gls[t])$.
It also commutes with the Cartan subalgebra $\fh_{\pqmn}$ of $\gls$ and hence acts on any weight space of a $\gls$-module.

For any $\Bsz$-module $V$, let $\big(\Bsz \big)_V$ denote the image of the Bethe algebra $\Bsz$ in $\End(V)$. We call $\big(\Bsz \big)_V$  the \emph{Bethe algebra} for $V$ with respect to $\z$.

\subsection{The Bethe duality of $(\gld, \gls)$} \label{Bethe-duality}

Let $\F:=\C[\x, \y, \et, \zet]$ be the polynomial superalgebra generated by $\xai$, $\eaj$, $\yar$, and $\zas$, for $i=1, \ldots, m$, $j=1, \ldots, n$, $r=1, \ldots, p$, $s=1, \ldots, q$, and $a=1, \ldots, d$. Here $\xai$ and $\yar$ are even variables while $\eaj$ and $\zas$ are odd.
With an appropriate inner product, $\F$ may be regarded as the Fock space of $d(p+m)$ bosonic and $d(q+n)$ fermionic oscillators (cf. \cite{CL03, LZ06}).
Let $\cD:={\mathbb D}[\x, \y, \et, \zet]$ be the corresponding \emph{Weyl superalgebra}, i.e., $\cD$ is the associative unital superalgebra generated by
$\xai, \, \eaj, \, \yar, \, \zas$
as well as their derivatives
$$
\pxai:=\frac{\pa}{\pa x_i^a}, \quad \peaj:=\frac{\pa}{\pa\eta_j^a}, \quad \pyar:=\frac{\pa}{\pa y_r^a}, \quad \pzas:=\frac{\pa}{\pa\zeta_r^a}
$$
for $1\le i \le m$, $1\le j \le n$, $1\le r \le p$, $1\le s \le q$, and $1 \le a \le d$.
The superalgebra $\cD$ acts naturally on $\F$, and so it is a subalgebra of $\End(\F)$.

According to \cite[pp. 789--790]{CLZ}, there is a superalgebra homomorphism $\phi : U(\gld) \longrightarrow \cD$, given by
\begin{equation}\label{phi}
e_{ab} \mapsto \sum_{i=1}^m x_i^a \pxbi+\sum_{j=1}^n\eta_j^a \pebj-\sum_{r=1}^p y_{r}^{b} \pyar -\sum_{s=1}^q\zeta_s^{b} \pzas,
\end{equation}
for $1 \le a, b \le d$,
and a superalgebra homomorphism $\varphi : U(\gls) \longrightarrow \cD$, given by
\begin{equation} \label{ovphi}
{\everymath={\displaystyle}
\begin{array}{cc}
E^r_{r'} \mapsto -\sum_{a=1}^d \pyar y_{r'}^a, \quad
&  E^r_{p+s} \mapsto \sum_{a=1}^d \pyar \zeta_s^a, \\
E^{p+s}_r \mapsto -\sum_{a=1}^d \pzas y_r^a, \quad
&E^{p+s}_{p+s'} \mapsto \sum_{a=1}^d \pzas \zeta_{s'}^a, \\
E^r_{p+q+i} \mapsto \sum_{a=1}^d \pyar  \pxai,\quad
&E^r_{p+q+m+j} \mapsto \sum_{a=1}^d \pyar \peaj,\\
E^{p+s}_{p+q+i} \mapsto \sum_{a=1}^d  \pzas \pxai,\quad
&E^{p+s}_{p+q+m+j} \mapsto \sum_{a=1}^d  \pzas \peaj, \\
E^{p+q+i}_r\mapsto -\sum_{a=1}^d x^a_iy_{r}^a,\quad
& E^{p+q+i}_{p+s} \mapsto \sum_{a=1}^dx^a_i\zeta_s^a,\\
E^{p+q+m+j}_r\mapsto -\sum_{a=1}^d \eta^a_jy_{r}^a,\quad
& E^{p+q+m+j}_{p+s} \mapsto \sum_{a=1}^d\eta^a_j\zeta_s^a, \quad\\
E^{p+q+i}_{p+q+i'} \mapsto \sum_{a=1}^d x_i^a\pxaip,\quad
& E^{p+q+i}_{p+q+m+j} \mapsto  \sum_{a=1}^d x_i^a \peaj,   \\
E^{p+q+m+j}_{p+q+i} \mapsto  \sum_{a=1}^d\eta_j^a \pxai,\quad
 & E^{p+q+m+j}_{p+q+m+j'} \mapsto \sum_{a=1}^d\eta_{j}^a \peajp.
\end{array}}
\end{equation}
where $1\le r, r^\prime \le p$, $1\le s, s^\prime \le q$, $1\le i, i' \le m$, and $1\le j, j' \le n$.
These maps induce a $\gld \times \gls$-action on $\F$.

Let $\bpi: \gld \longrightarrow \gld$ be the $*$-structure on $\gld$ defined by
$$
\bpi(e_{ab})=e_{ba} \qquad \mbox{for $1 \le a, b \le d$.}
$$
Note that any unitarizable highest weight $\gld$-module with respect to $\bpi$ is finite-dimensional.
Recall the $*$-structure $\sig_p$ on $\gls$ defined by \eqref{star-gl}.
The maps $\bpi$ and $\sig_p$ give rise to a $*$-structure $(\bpi, \sig_p)$ on $\gld \times \gls$.
It is well-known that $\F$ decomposes into a direct sum of unitarizable $\gld\times \gls$-modules (with respect to $(\bpi, \sig_p)$). This is called the \emph{Howe duality of $(\gld, \gls)$} (\cite{CLZ}). Such unitarizable modules are, in general, infinite-dimensional.
Taking $p=q=0$, the maps $\phi$ and $\varphi$ specialize to the superalgebra homomorphisms $U(\gld) \longrightarrow {\mathbb D}[\x, \et]$ and $U(\glmn) \longrightarrow {\mathbb D}[\x, \et]$, respectively, and the duality reduces to the Howe duality of $(\gld,  \glmn)$, which yields a decomposition of $\C[\x, \et]$ into a direct sum of finite-dimensional $\gld\times \glmn$-modules (\cite{CW, Se01-1, Se01-2}). Assuming further that $n=0$, it gives the Howe duality of $(\gld,  \gl_m)$ (\cite{H}).

We simplify notations for the rest of this subsection. Let
$$
x^a_{m+j}=\eta^a_j,  \quad \pa_{x^a_{m+j}}=\peaj, \quad y^a_{p+s}=\xi^a_s, \quad  \pa_{y^a_{p+s}}=\pzas,
$$
where $1\le j \le n$ and $1\le s \le q$.
Fix two sequences
$$
\w:=(w_1,\ldots,w_d) \in \C^d \qquad \text{and} \qquad \z:=(z_1,\ldots, z_{p+q+m+n}) \in \C^{p+q+m+n}.
$$
The map $\phi$ induces a superalgebra homomorphism
$
\phi_\z : U(\gld[t])\longrightarrow \cD
$
defined by
\begin{equation}\label{phi_z}
e_{ab}[k] \mapsto  -\sum_{r=1}^{p+q} z_r^k \ybr \pyar + \sum_{i=1}^{m+n} z_{p+q+i}^k \xai \pxbi,
\end{equation}
for $1 \le a, b \le d$ and $k \in \Zp$, and $\varphi$ induces a superalgebra homomorphism $\varphi_\w : U(\gls[t]) \longrightarrow \cD$ defined by
\begin{equation}\label{ovphi_w}
{\everymath={\disp}
\begin{array}{cll}
E^r_s [k] &\mapsto \sum_{a=1}^{d} (-1)^{|s|+1} w_a^k \pyar \yas, \quad
E^r_{p+q+j} [k] &\mapsto \sum_{a=1}^{d} w_a^k \pyar \pxaj, \\
E^{p+q+i}_s [k]&\mapsto\sum_{a=1}^{d} (-1)^{|s|+1} w_a^k \xai \yas,\quad
E^{p+q+i}_{p+q+j} [k]  &\mapsto \sum_{a=1}^{d} w_a^k \xai \pxaj,
\end{array}}
\end{equation}
for $1 \le r, s \le p+q$, $1 \le i, j \le m+n$, and $k \in \Zp$.
Since the Bethe algebras $\Bdw$ and $\Bsz$ are subalgebras of $U(\gld[t])$ and $U(\gls[t])$, respectively, the maps $\phi_\z$ and $\varphi_\w$ induce an action of $\Bdw$ on $\F$ and an action of $\Bsz$ on $\F$, respectively.

In view of the Howe duality, it is natural to ask if any duality exists in the context of Bethe algebras as well.
When $p=q=n=0$, a duality phenomenon is noticed in \cite{TL} (see also \cite{TV}).
In \cite{MTV09}, the precise statement is made and proved.
So far, various generalizations have emerged (see \cite{HM, TU, VY}). We will show that the desired duality exists between $\Bdw$ and $\Bsz$ in a moment.

The maps $\phi_\z$ and $\varphi_\w$ extend naturally to the superalgebra homomorphisms
$$
\phi_\z: U(\gld[t]) \zpz \longrightarrow \cD \zpz
$$
and
$$
\varphi_\w: U(\gls[t])\zpz \longrightarrow \cD \zpz,
$$
respectively.
Furthermore, there is an anti-involution $\omega : \cD \longrightarrow \cD$ defined by
\begin{equation} \label{anti-inv}
 \xai \mapsto  \pxai,  \quad \pxai \mapsto \xai,  \quad \yar \mapsto \pyar,  \quad \pyar  \mapsto \yar,
\end{equation}
for $1 \le i \le m+n$, $1 \le r  \le p+q$, and $1 \le a \le d$.
In light of \propref{omega}, the map $\omega$ extends naturally to an anti-involution
$$
\omega : \cD \zpz \longrightarrow \cD \zpz.
$$

For $u=z$ or $\pz$, let
$$
R_{p|q}(u)=\frac{(u-z_1)\ldots(u-z_p)}{(u-z_{p+1})\ldots(u-z_{p+q})}
$$
and
$$
R^\prime_{m|n}(u)=\frac{(u-z_{p+q+1})\ldots(u-z_{p+q+m})}{(u-z_{p+q+m+1})\ldots(u-z_{p+q+m+n})}.
$$
We have the following proposition.
We defer its proof until Appendix \ref{Bers-cdet} to avoid disrupting the flow of the paper.

\begin{prop} \label{omega-Ber}
Let $\s=(0^p, 1^q, 0^m, 1^n) \in \cS_\pqmn$. Then
$$
(z-w_1)\ldots(z-w_d) \,  \varphi_\w  \Bers \big( \Lsz \big) = \omega \lb R^\prime_{m|n}(z) \cdot \phi_\z  \cdet (\Ldw) \cdot R_{p|q}(z) \rb.
$$
\end{prop}

The Cartan anti-involution $\varpi$ on $\gld$, defined by \eqref{Cartan-anti}, gives rise to an anti-involution $\varpi: U(\gld[t]) \longrightarrow U(\gld[t])$, defined by
$$
\hspace{1.5cm} \varpi(A[k])=\varpi(A)[k] \qquad \mbox{for $A \in \gld$ and $k \in \Zp$}.
$$

\begin{prop}[{\cite[Proposition 8.4]{MTV06}}] \label{C-inv-Gau}
The anti-involution $\varpi: U(\gld[t]) \longrightarrow U(\gld[t])$ fixes all elements of $\Bdw$.
\end{prop}

\begin{prop} \label{omega-inv}
For any $\bb \in \Bdw$, $\omega \big( \phi_\z(\bb) \big)=\phi_\z(\bb)$.
\end{prop}
\begin{proof}
It is straightforward to verify that the diagram
\begin{eqnarray*} \label{com-omega}
\CD   U(\gld[t])
@> \phi_\z>>\cD \\
 @V \varpi VV @VV \omega V \\
U(\gld[t])
  @> \phi_\z>>\cD\\
\endCD
\end{eqnarray*}
commutes.
This, together with \propref{C-inv-Gau}, proves the proposition.
\end{proof}

The following theorem is called the \emph{Bethe duality of $(\gld, \gls)$}.
It not only relates the coefficients of $\phi_\z  \cdet (\Ldw)$ to those of $\varphi_\w  \Bers \big( \Lsz \big)$, but also gives an equality of algebras.

\begin{thm}\label{Bethe_duality}
Let $\s=(0^p, 1^q, 0^m, 1^n) \in \cS_\pqmn$.
We may write
$$
R_{p|q}(z) \cdot R^\prime_{m|n}(z) \cdot \phi_\z  \cdet (\Ldw)=\sum_{\ell=0}^d \sum_{k=-\infty}^{p+m-q-n} b_{k,\ell}^{\w,\z} z^k\pz^\ell,
$$
where $b_{k,\ell}^{\w,\z}:=\phi_\z(b_{k,\ell}^\w)$ for some $b_{k,\ell}^\w \in \Bdw$. Then
\begin{eqnarray*}
&  &(z-w_1)\ldots(z-w_d)\  \varphi_\w  \Bers \big( \Lsz \big) \\
& & \hspace{3cm} =R_{p|q}(\pz) \left(\sum_{k=-\infty}^{p+m-q-n} \sum_{\ell=0}^{d} b_{k, \ell}^{\w,\z} z^\ell \pa_z^{k}  \right) R_{p|q}(\pz)^{-1}.
\end{eqnarray*}
Moreover, $\phi_\z (\Bdw)=\varphi_\w \big(\Bsz \big)$.
\end{thm}
\begin{proof}
By \propref{omega-Ber},
\begin{eqnarray*}
& & (z-w_1)\ldots(z-w_d) \, \varphi_\w \Bers \big( \Lsz \big)\\
& & \hspace{1.5cm} = R_{p|q}(\pz)  \cdot \omega \left( R^\prime_{m|n}(z) \cdot \phi_\z  \cdet (\Ldw) \right) \! \\
& & \hspace{1.5cm} = R_{p|q}(\pz) \cdot \omega \left(R_{p|q}(z) \cdot  R^\prime_{m|n}(z) \cdot \phi_\z  \cdet (\Ldw) \right) \cdot R_{p|q}(\pz)^{-1} \\
& & \hspace{1.5cm} =R_{p|q}(\pz) \left(\sum_{k=-\infty}^{p+m-q-n} \sum_{\ell=0}^{d} \omega\big(b_{k, \ell}^{\w,\z} \big)  z^\ell \pa_z^{k}  \right) R_{p|q}(\pz)^{-1}.
\end{eqnarray*}
By \propref{omega-inv}, $\omega\big(b_{k, \ell}^{\w,\z} \big) = b_{k, \ell}^{\w,\z}$
for all $k \le p+m-q-n$ and $0 \le \ell \le d$. This proves the first assertion.
The last assertion clearly follows from the first one.
\end{proof}

\begin{rem} \label{HM-rem}
We have learned from an anonymous expert that the last assertion of \thmref{Bethe_duality} can be obtained directly from the results of \cite{HM}.
An outline of his approach is given as follows.
Let $\pi$ be an automorphism on $\cD$ defined by
$$
 \xai \mapsto \xai,  \quad \pxai \mapsto \pxai,  \quad \yar \mapsto (-1)^{|r|+1} \pyar,  \quad \pyar  \mapsto \yar,
$$
for $1 \le i \le m+n$, $1 \le r  \le p+q$, and $1 \le a \le d$,
and let $\vartheta$ be an automorphism on $U(\gld[t])$ defined by
$$
e_{ab}[k] \mapsto e_{ab}[k]-\de_{a,b} \sum_{r=1}^{p+q}  (-1)^{|r|} z_r^k
$$
for $1 \le a, b \le d$ and $k \in \Zp$.
Note that $\vartheta$ restricts to an automorphism on $\Bdw$.
As a consequence of \cite[Theorem 5.2]{HM},
$$(\pi \circ \phi_\z \circ \vartheta) (\Bdw)=(\pi \circ \varphi_\w) (\Bsz),$$
which is equivalent to the last assertion of \thmref{Bethe_duality}.
\end{rem}

\subsection{An application of the Bethe duality} \label{app}

The Weyl superalgebra $\cD$ is a subalgebra of $\End(\F)$.
For $\w \in \C^d$ and $\z \in \C^{p+q+m+n}$, \thmref{Bethe_duality} implies that the image of $\Bdw$ in $\End(\F)$ under $\phi_\z$ coincides with the image of $\Bsz$ in $\End(\F)$ under $\varphi_\w$.
This makes it possible to study the action of $\Bsz$ on $\F$ via the action of $\Bdw$ on $\F$, and vice versa.

Let
$$
\dds=\sum_{a=1}^{d} \lb  \sum_{i=1}^m x_i^a \pxai+\sum_{j=1}^n\eta_j^a \peaj-\sum_{r=1}^p y_{r}^a \pyar -\sum_{s=1}^q\zeta_s^a \pzas \rb \!.
$$
For any monomial $f \in \F$, we have $\dds(f)=k f$ for some $k \in \Z$.
We define the degree of $f$ to be $k$.
In particular, the degrees of ${x^a_i}$ and ${\eta^a_j}$ are $1$ while the degree of $y^a_r$ and $\zeta^a_s$ are $-1$.
We call $\dds$ the degree operator on $\F$.

For $1\le r \le p$, $1\le s \le q$, $1\le i \le m$, and $1\le j \le n$,
let $\F_{(r)}=\C \big[y^1_r,\ldots, y^d_r \big]$, $\F_{(p+s)}=\C \big[\zeta^1_s,\ldots, \zeta^d_s \big]$, $\F_{(p+q+i)}=\C \big[x^1_i,\ldots, x^d_i \big]$,
and $\F_{(p+q+m+j)}= \C \big[\eta^1_j,\ldots, \eta^d_j \big]$.
It is evident that as $\gld$-modules,
\begin{equation} \label{gld-dec}
\F =  \F_{(1)} \otimes \cdots \otimes \F_{(p+q+m+n)}.
\end{equation}
For $1 \le a \le d$, let
$\F^{(a)}=\C \big[ x^a_{i}, \eta^a_{j}, y^a_{r}, \zeta^a_{s} \, \big| \, 1\le i \le m, 1\le j \le n, 1\le r \le p, 1\le s \le q \big]$.
As $\gls$-modules,
\begin{equation} \label{gls-dec}
\F =  \F^{(1)} \otimes \cdots \otimes \F^{(d)}.
\end{equation}

For any weight $\mu \in \fh_{\pqmn}^*$ of a $\gls$-module $M$, we denote by $M_\mu$ the $\mu$-weight space of $M$.
The Cartan subalgebra $\fh_{\pqmn}$ of $\gls$ acts on $\F$ as follows (see \eqref{ovphi}): $E^i_i$ acts on $\F$ by
$-\sum_{a=1}^{d} y^a_i  \pa_{y^a_i}- d, $
for $1 \le i \le p$,
and by
$-\sum_{a=1}^{d} \zeta^a_{i-p}  \pa_{\zeta^a_{i-p} }+ d,$
for $p+1 \le i \le q$.
Also, $E^{p+q+i}_{p+q+i}$ acts on $\F$ by
$\sum_{a=1}^{d} x^a_i \pa_{x^a_i},$
for $1 \le i \le m$, and by
$\sum_{a=1}^{d} \eta^a_{i-m} \pa_{\eta^a_{i-m}},$
for $m+1 \le i \le m+n$.
Clearly, there is a weight space decomposition
\begin{equation} \label{Fmu}
\F=  \bigoplus_{\mu\in \fh_{\pqmn}^{*}}  \F_\mu.
\end{equation}

For $\xi \in \fh_{\pqmn}^{*}$, we denote by $L_\pqmn (\xi)$ the irreducible highest weight $\gls$-module with highest weight $\xi$ with respect to the Borel subalgebra $\mf{b}_{\pqmn}$.
Let
$$
\opq=\sum_{i=1}^{p} \ep_i-\sum_{j=1}^{q} \ep_{p+j} \qquad \text{and} \qquad \op=\mathbbm{1}_{p|0}.
$$
Suppose $p\not=0\not=m$.
For $r \in \Z$, let
\begin{equation} \label{depth1}
V_{r} =
\begin{cases}
\disp{ L_{\pqmn} \big( r \ep_{p+q+1}- \opq \big)} &\quad \mbox{if $r \ge 0$,}\\
\disp{L_{\pqmn} \lb -\sum_{i=q+r+1}^q \ep_{p+i}-\opq \rb} &\quad \mbox{if $-q \le r<0$ and $q \not=0$,}\\
\disp{L_{\pqmn}\lb (r+q) \ep_p-\sum_{i=1}^q \ep_{p+i}-\opq \rb} &\quad \mbox{if $r<-q$ and $q \not=0$,}\\
\disp{L_{\pqmn} \big(r \ep_p- \op \big)}   &\quad  \mbox{if $r<0$ and $q=0$}.
\end{cases}
\end{equation}
The space $V_r$ is an infinite-dimensional unitarizable highest weight $\gls$-module (with respect to $\sig_p$), and the highest weight of $V_r$, appearing in \eqref{depth1}, corresponds to a generalized partition of depth 1; see \cite[Section 2.5]{CCL25-1} and also \cite{CLZ}. (Such a partition is called a generalized partition of length 1 in \cite{CLZ}.)

For $a=1, \ldots, d$, let $\ga_a \in \Z$. 
If $\ga_a \ge 0$, then we may identify $V_{\ga_a}$ with the $\gls$-submodule of the polynomial superalgebra $\F^{(a)}$ generated by  the highest weight vector $(x^a_1)^{\ga_a}$.
If $\ga_a < 0$, then we may identify $V_{\ga_a}$ with the $\gls$-submodule of $\F^{(a)}$ generated by the highest weight vector
$$
\begin{cases}
\zeta^a_{q+\ga_a+1} \ldots \zeta^a_q &\quad \text{if } \mbox{$\ga_a \ge - q$ and $q \not=0$,}\\
(y^a_p)^{-\ga_a-q} \zeta^a_1 \ldots \zeta^a_q &\quad \text{if } \mbox{$\ga_a < -q$ and $q \not=0$,}\\
(y^a_p)^{-\ga_a}   &\quad \text{if } q=0.
\end{cases}
$$
In any case, $V_{\ga_a}$ is the subspace of $\F^{(a)}$ spanned by all monomials of degree $\ga_a$.
Via \eqref{gls-dec}, the tensor product $V_{\ga_1} \otimes \cdots \otimes V_{\ga_d}$ is a $\gls$-submodule of $\F$.
We thus obtain a direct sum decomposition of $\gls$-modules:
\begin{equation} \label{F-decomp}
\F= \bigoplus_{\ga_1, \ldots, \ga_{d} \in \Z}  V_{\ga_1} \otimes \cdots \otimes V_{\ga_d}.
\end{equation}

Let $\fh_d$ denote the standard Cartan subalgebra of $\gld$.
The action of $\fh_d$ on $\F$ is determined by the action of $e_{aa}$ on $\F$, which is given by
$$
 \sum_{i=1}^m x_i^a \pxai+\sum_{j=1}^n\eta_j^a \peaj-\sum_{r=1}^p y_{r}^a \pyar -\sum_{s=1}^q\zeta_s^a \pzas,
$$
for $1 \le a \le d$ (see \eqref{phi}).
There is a weight space decomposition
\begin{equation} \label{Fga}
\F=  \bigoplus_{\ga \in \fh_{d}^{*}}  \F_\ga.
\end{equation}

For $1 \le i \le p+q$ and $k \in-\Zp$, let $\F_{(i)}^{k}$ be the subspace of $\F_{(i)}$ spanned by all monomials of degree $k$.
For $1 \le j \le m+n$ and $l \in \Zp$, let $\F_{(p+q+j)}^{l}$ be the subspace of $\F_{(p+q+j)}$ spanned by all monomials of degree $l$.
Evidently, $\F_{(i)}^{k}$ and $\F_{(p+q+j)}^{l}$ are finite-dimensional $\gld$-modules.
We have the following direct sum decompositions of $\gld$-modules:
$$
\F_{(i)} =\bigoplus_{k \in-\Zp} \F_{(i)}^{k}
\qquad \text{and} \qquad
\F_{(p+q+j)} = \bigoplus_{l \in\Zp} \F_{(p+q+j)}^{l}.
$$
It is well known that $\F_{(i)}^{k}$ and $\F_{(p+q+j)}^{l}$ are irreducible $\gld$-modules for $k \in-\Zp$ and $l \in \Zp$ (see, for instance, \cite[Theorem 3.3]{CLZ} by setting exactly one of $p,q,m,n$ (given there) to 1 and the others to 0).
By \eqref{gld-dec}, $\F$ decomposes into a direct sum of $\gld$-modules
\begin{equation} \label{Fmub}
\F=\bigoplus_{\substack{k_1, \ldots, k_{p+q}\in -\Zp,\\ k_{p+q+1},\ldots,k_{p+q+m+n} \in \Zp}} \F_{(1)}^{k_1}\otimes \cdots \otimes \F_{({p+q+m+n})}^{k_{p+q+m+n}}.
\end{equation}

\begin{prop} \label{sws}
For $\ga_1, \ldots, \ga_d \in \Z$, let $\ga=\sum_{a=1}^{d}  \ga_a \ep_a \in \fh_{d}^*$.
Let $\mu \in \fh_{\pqmn}^*$ be any weight of $V_{\ga_1} \otimes \cdots \otimes V_{\ga_d}$.
Then
$$
(V_{\ga_1} \otimes \cdots \otimes V_{\ga_d})_\mu=\big(\F_{(1)}^{\ov\mu_1} \otimes \cdots \otimes \F_{({p+q+m+n})}^{\ov\mu_{p+q+m+n}} \big)_{\ga},
$$
where
\begin{equation} \label{ovmu}
\ov\mu_i=\begin{cases}
\, \mu(E^i_i)+(-1)^{|i|} d &   \mbox{if} \ \ 1\le i \le p+q ;\\
\, \mu(E^i_i) \ \ & \mbox{if} \ \ p+q+1\le i \le p+q+m+n.
\end{cases}
\end{equation}
\end{prop}

\begin{proof}
By \eqref{F-decomp} and \eqref{Fga}, we have
$$
 \F_{\ga}=V_{\ga_1} \otimes \cdots \otimes V_{\ga_d}.
$$
Also, \eqref{Fmu} and \eqref{Fmub} yield
$$
\F_\mu=\F_{(1)}^{\ov\mu_1} \otimes \cdots \otimes \F_{({p+q+m+n})}^{\ov\mu_{p+q+m+n}}.
$$
Consequently,
$$
(V_{\ga_1} \otimes \cdots \otimes V_{\ga_d})_\mu=\F_\ga \cap\F_\mu = \big(\F_{(1)}^{\ov\mu_1} \otimes \cdots \otimes \F_{\ov\mu_{p+q+m+n}}^{(p+q+m+n)} \big)_{\ga},
$$
as asserted.
\end{proof}

For $\gls$-modules $M_1, \ldots, M_d$ and $\w=(w_1, \ldots, w_d) \in \C^d$, we write
$$
\bn M= M_1 \otimes \cdots \otimes M_d   \qquad \text{and} \qquad \bn M(\w)=M_1(w_1) \otimes \cdots \otimes M_d (w_d).
$$

\begin{thm} \label{Bethe-unitary}
Suppose $p\not=0\not=m$.
For $a=1, \ldots, d$, let $L_{a}=V_{\ga_a}$ for some $\ga_a \in \Z$.
For any weight $\mu \in \fh_{\pqmn}^*$ of $\un L$, $\z \in \Xpqmn$, and $\w \in \Xd$, we have:
\begin{enumerate}[\normalfont(i)]

\item $\un L(\w)_{\mu}$ is a cyclic $\Bsz$-module.

\item $\big(\Bsz \big)_{\un L(\w)_{\mu}}$ is a Frobenius algebra.

\item  $\big(\Bsz \big)_{\un L(\w)_{\mu}}$ is a maximal commutative subalgebra of $\End(\un L(\w)_{\mu})$ of dimension $\dim \! \left(\un L(\w)_{\mu} \right)$.

\item Every eigenspace of $\big(\Bsz \big)_{\un L(\w)_{\mu}}$ is one-dimensional.

\end{enumerate}
Moreover, $\Bsz$ is diagonalizable with a simple spectrum on $\un L(\w)_{\mu}$ for generic $\z$ and generic $\w$.

\end{thm}

\begin{proof}
Let $\mu \in \fh_{\pqmn}^*$ be a weight of $\un L$.
For $i \in \I$, let $N_i=\F_{(i)}^{\ov\mu_i}$, where $\ov\mu_i$ is defined as in \eqref{ovmu}.
Write $\z=(z_1, \ldots, z_{p+q+m+n})$.
We set $\bn N= N_1 \otimes \cdots \otimes N_{p+q+m+n}$ and $\bn N(\z)=N_1(z_1)  \otimes \cdots \otimes  N_{p+q+m+n} (z_{p+q+m+n})$.
In view of \thmref{Bethe_duality} and \propref{sws},
\begin{equation} \label{eq-act}
\big(\Bsz \big)_{\un L(\w)_{\mu}}=(\Bdw)_{ \bn N(\z)_\ga}.
\end{equation}
By \thmref{g-cyclic}, $\bn N(\z)_\ga$ is a cyclic $\Bdw$-module as $\Bdw$ preserves weight spaces, and by \thmref{Frob}, $(\Bdw)_{ \bn N(\z)_\ga}$ is a Frobenius algebra.
Hence, (i) and (ii) follow from \eqref{eq-act}.
Also, (iii) and (iv) are immediate consequences of \cite[Lemma 1.3]{Lu} (cf. \remref{max-r}). The last statement follows from \eqref{eq-act} and \thmref{g-diag}.
\end{proof}

\begin{rem}

A similar statement corresponding to $p=0$ or $m=0$ can be formulated. The details will be left to the reader.

\end{rem}

\thmref{Bethe-unitary} suggests that the super analog of \conjref{cyclic-diag-conj} should be true for the general linear Lie superalgebra.

\begin{conj} \label{cyclic-diag-gls}
Suppose $p$, $q$, $m$, and $n$ are not all zero. For $a=1, \ldots, d$, let $L_a$ be a unitarizable highest weight $\gls$-module (with respect to $\sig_p$).
For any weight $\mu \in \fh_{\pqmn}^*$ of $\un L$, $\z \in \Xpqmn$, and $\w \in \Xd$, parts (i)--(iv) of \thmref{Bethe-unitary} hold for the Bethe algebra $(\Bsz)_{\un L(\w)_{\mu}}$.
Moreover, $\Bsz$ is diagonalizable with a simple spectrum on $\un L(\w)_{\mu}$ for generic $\z$ and generic $\w$.

\end{conj}

\begin{rem}
An analog of \conjref{cyclic-diag-gls} when the entries of $\z$ are all zero is studied in \cite{CCL25-2}.
\end{rem}

Suppose $p\not=0\not=m$.
For $r \in \Z$, let
$$
W_r=
\begin{cases}
L_{p+m|0} (r \ep_{p+1} - \op)  &\quad  \mbox{if $r \ge 0$,}\\
L_{p+m|0} (r \ep_p- \op) &\quad \mbox{if $r < 0$}.
\end{cases}
$$
The following provides evidence for \conjref{cyclic-diag-conj}.

\begin{cor}  \label{evid}
For $a=1, \ldots, d$, let $L_{a}=W_{\ga_a}$ for some $\ga_a \in \Z$.
For any weight $\mu \in \fh_{p+m}^*$ of $\un L$, $\z \in \Xpm$, and $\w \in \Xd$, parts (i)--(iv) of \thmref{Bethe-unitary} hold for the Bethe algebra $(\Bpm)_{\un L(\w)_{\mu}}$.
Moreover, $\Bpm$ is diagonalizable with a simple spectrum on $\un L(\w)_{\mu}$ for generic $\z$ and generic $\w$.

\end{cor}

\begin{proof}
This is an immediate consequence of \thmref{Bethe-unitary} by taking $q=n=0$.
\end{proof}

\vskip 2mm
\noindent{\bf Acknowledgments.}
The first author was partially supported by NSTC grants 113-2115-M-006-010 and 114-2115-M-006-005.
The second author was partially supported by NSTC grant 112-2115-M-006-015-MY2.
The authors are grateful to an anonymous expert for valuable comments on \thmref{Bethe_duality} (see \remref{HM-rem}).

\vskip 7mm

\appendix
\section{The proof of \propref{omega-Ber}} \label{Bers-cdet}

The appendix is devoted to proving \propref{omega-Ber}.
We keep the notations defined in \secref{Duality}.
Let
$$
\mc C=\left[\delta_{a,b}(z-w_a)+\sum_{r=1}^{p+q}\frac{\yar\pybr}{\pz-z_r} -\sum_{i=1}^{m+n}\frac{ \xbi \pxai}{\pz-z_{p+q+i}}\right]_{a,b=1,\ldots,d}.
$$

\begin{lem} \label{omega_det}
$\omega\big( \phi_\z  \cdet (\Ldw) \big)=\cdet (\mc C)$.
\end{lem}

\begin{proof}
For $1 \le a, b \le d$,
\begin{equation}\label{phi_z_w}
\phi_\z \big(e_{ab}(z) \big)=-\sum_{r=1}^{p+q}\frac{\ybr \pyar}{z-z_r}+\sum_{i=1}^{m+n}\frac{\xai \pxbi}{z-z_{p+q+i}}.
\end{equation}
For $i \in \I \cup \{0\}$ and $1 \le a, b \le d$, we define
$$
F^{i}_{a,b}=\begin{cases}
\, \, \delta_{a, b}(\pz-w_a)  &     \mbox{if } \, i=0,\\
\disp{\, \, \frac{\ybi \pyai}{z-z_i}}  &   \mbox{if } \,  1 \le i \le p+q,\\
\disp{ \,\frac{-x^a_{i-p-q}\pa_{x^b_{i-p-q}}}{z-z_i}}  &    \mbox{otherwise},
\end{cases}
$$
and
$$
G^{i}_{a,b}=\begin{cases}
\, \, \delta_{a, b}(z-w_a) &  \mbox{if } \, i=0,\\
\disp{\, \, \frac{\yai \pybi}{\pz-z_i}} & \mbox{if } \,  1 \le i \le p+q,\\
\disp{ \,\frac{-x^b_{i-p-q}\pa_{x^a_{i-p-q}}}{\pz-z_i}} & \mbox{otherwise}.
\end{cases}
$$
We see that
\begin{equation} \label{FG}
\omega (F^i_{a, b})=G^i_{a, b}.
\end{equation}

Note that $\Ldw$ is an amply invertible Manin matrix of type $(0^d)$. Let $\pi \in \fS_d$ be the permutation given by $\pi(i)=d+1-i$ for $i=1, \ldots, d$. Clearly, $\pi^2=1$.
By \eqref{phi_z_w}, \propref{CFR} and \propref{Ber-sig-L}, we obtain
\begin{eqnarray*}
\phi_\z  \cdet (\Ldw) &=&  \phi_\z  \cdet \big((\Ldw)^{\pi} \big) \\
&=&\sum_{\si\in\fS_d}\sum_{i_1,\ldots,i_d=0}^{p+q+m+n} (-1)^{l(\si)} F^{i_d}_{\pi(\si(1)), \pi(1)} \ldots F^{i_1}_{\pi(\si(d)), \pi(d)}\\
&=&\sum_{\si\in\fS_d}\sum_{i_1,\ldots,i_d=0}^{p+q+m+n} (-1)^{l(\si)} F^{i_d}_{\pi \si(1), d} \ldots F^{i_1}_{\pi \si(d), 1}\\
&=&\sum_{\si\in\fS_d}\sum_{i_1,\ldots,i_d=0}^{p+q+m+n} (-1)^{l(\si)} F^{i_d}_{\si(d),d} \ldots F^{i_1}_{\si(1), 1}.
\end{eqnarray*}
Here the last equality holds since $\pi^{-1} \fS_d \pi =\fS_d$ and $l(\pi^{-1} \si \pi)=l(\si)$ for all $\si \in \fS_d$.
Thus,
$$
\omega\big(\phi_\z  \cdet (\Ldw) \big)=\sum_{\si\in\fS_d}\sum_{i_1,\ldots,i_d=0}^{p+q+m+n} (-1)^{l(\si)} \omega \! \left(F^{i_1}_{\si(1),1} \right) \ldots  \omega \! \left(F^{i_d}_{\si(d),d}  \right).
$$
In view of \eqref{FG}, we deduce that $\omega\big(\phi_\z  \cdet (\Ldw) \big)=\cdet (\mc C)$.
\end{proof}

Now we are ready to prove \propref{omega-Ber}.
For $1 \le r, s \le p+q$ and $1 \le i, j \le m+n$, we have
\begin{equation}\label{ovphi_w_z}
{\everymath={\disp}
\begin{array}{cll}
\varphi_\w \big(E^r_s(z) \big) \hspace{-1cm} &=\sum_{a=1}^{d}  \frac{(-1)^{|s|+1} \pyar \yas}{z-w_a}, \quad
\varphi_\w \big(E^r_{p+q+j}(z) \big) \hspace{-0.3cm} &=\sum_{a=1}^{d}  \frac{\pyar \pxaj}{z-w_a}, \\
\varphi_\w \big(E^{p+q+i}_s(z) \big) \hspace{-0.3cm} &=\sum_{a=1}^{d}  \frac{(-1)^{|s|+1} \xai \yas}{z-w_a}, \quad
\varphi_\w \big(E^{p+q+i}_{p+q+j}(z) \big) \hspace{-0.3cm}  &=\sum_{a=1}^{d}  \frac{\xai \pxaj}{z-w_a}.
\end{array}}
\end{equation}
Let
$$
D=\Big[\delta_{a,b}(z-w_a) \Big]_{a,b=1,\ldots,d} \qquad \text{and} \qquad  D^\prime=\Big[\de_{i,j}(\pz-z_i) \Big]_{i,j\in \I}.
$$
Consider the $(m+n) \times d$ matrices
$$
X:=\big[ (-1)^{|p+q+i|} \xai \big]^{a=1, \ldots, d}_{i=1, \ldots, m+n} \qquad \text{and} \qquad P_X:=\big[ \pxai \big]^{a=1, \ldots, d}_{i=1, \ldots, m+n}
$$
and the $(p+q) \times d$ matrices
$$
Y:=\big[ (-1)^{|r|+1} \yar \big]^{a=1, \ldots, d}_{r=1, \ldots, p+q} \qquad \text{and} \qquad  P_Y:=\big[  (-1)^{|r|} \pyar \big]^{a=1, \ldots, d}_{r=1, \ldots, p+q}.
$$
Write
$$
T=
\begin{bmatrix}
Y^t \hspace{2mm} P_X^t
\end{bmatrix}
\qquad \text{and}  \qquad
T^\prime=
\begin{bmatrix} P_Y  \vspace{3mm}  \\ X
\end{bmatrix}.
$$
By \eqref{ovphi_w_z}, we get
$$
\varphi_\w  \Bers \big( \Lsz \big)=\Bers (D^\prime- T^\prime D^{-1} T).
$$

Let
$$
\fL=
\begin{bmatrix}
D & T \vspace{2mm}
 \\ T^\prime & D^\prime
\end{bmatrix}.
$$
It is straightforward to check that $\fL$ is an amply invertible Manin matrix of type $\hat{\s}:=(0^d, \s) \in \cS_{d+p+m|q+n}$.
By \propref{block}, we have
$$
\Ber^{\hat{\s}} (\fL)=(z-w_1)\ldots(z-w_d)\,  \varphi_\w \Bers \big(\Lsz \big).
$$
As $D- T (D^\prime)^{-1} T^\prime$ is a sufficiently invertible Manin matrix of type $(0^d)$, \propref{block2} and \propref{CFR} imply that
\begin{eqnarray*}
\Ber^{\hat{\s}} (\fL) &=& \Bers (D^\prime) \cdot \cdet (D- T (D^\prime)^{-1} T^\prime)\\
&=& R_{p|q}(\pz) \cdot R^\prime_{m|n}(\pz) \cdot \cdet (\mc C'),
\end{eqnarray*}
where
$$
\mc C':= \left[\delta_{a,b}(z-w_a)+\sum_{r=1}^{p+q}\frac{\yar \pybr}{\pz-z_r}-\sum_{i=1}^{m+n}\frac{(-1)^{|p+q+i|}\pxai \xbi}{\pz-z_{p+q+i}}  \right]_{a,b=1,\ldots,d}.
$$

We claim that
$$
R^\prime_{m|n}(\pz) \cdot \cdet (\mc C')= \cdet (\mc C) \cdot R^\prime_{m|n}(\pz).
$$
This yields
$$
\Ber^{\hat{\s}} (\fL) = R_{p|q}(\pz) \cdot \cdet (\mc C) \cdot R^\prime_{m|n}(\pz),
$$
which proves \propref{omega-Ber} by \lemref{omega_det}.

It remains to prove the claim. We only need to show that
\begin{equation} \label{switch}
R^\prime_{m|n}(\pz) \cdot \mc C'_{a,b} =\mc C_{a,b} \cdot R^\prime_{m|n}(\pz), \qquad \mbox{for $1 \le a, b \le d$.}
\end{equation}
Here, for instance, $\mc C_{a,b}$ denotes the $(a,b)$ entry of $\mc C$.
Since
\begin{equation} \label{px-x}
\pxai \xbi =(-1)^{|p+q+i|} \xbi \pxai +\de_{a,b}, \qquad \mbox{for $1 \le a, b \le d$},
\end{equation}
the off-diagonal entries of $\mc C'$ coincide with those of $\mc C$.
For $1 \le i \le m+n$, the term $\pz-z_{p+q+i}$ commutes with $\pxaj \xbj$ and $\yar \pybr$ for all $1 \le j \le m+n$, $1 \le r \le p+q$, and $1 \le a, b \le d$.
In particular, $R^\prime_{m|n}(\pz)$ commutes with any off-diagonal entry of $\mc C'$. This proves \eqref{switch} for $a \not=b$.
Moreover, for $1 \le i \le m+n$ and $1 \le a \le d$,
\begin{equation} \label{pz-z}
(\pz-z_{p+q+i}) (z-w_a)= (z-w_a)(\pz-z_{p+q+i})+1.
\end{equation}
If $1 \le i \le m$, we see from \eqref{px-x} and \eqref{pz-z} that
$$
{(\pz-z_{p+q+i})} \left(z-w_a-\frac{\pxai \xai}{\pz-z_{p+q+i}}\right)=\left(z-w_a-\frac{\xai\pxai}{\pz-z_{p+q+i}}\right)(\pz-z_{p+q+i}).
$$
Similarly, if $m+1 \le i \le m+n$, we have
$$
{(\pz-z_{p+q+i})^{-1}} \left(z-w_a+\frac{\pxai \xai}{\pz-z_{p+q+i}}\right)=\left(z-w_a-\frac{\xai\pxai}{\pz-z_{p+q+i}}\right)(\pz-z_{p+q+i})^{-1}.
$$
The above discussion proves \eqref{switch} for $a=b$. Thus, the proof of \eqref{switch} is now complete.

\end{document}